\documentclass[11pt]{article}
\usepackage{geometry}
\usepackage{amssymb}
\usepackage{amsmath}
\usepackage{float}
\usepackage{color}
\usepackage{graphicx}
\usepackage{subfigure}
\usepackage{tabu}
\usepackage{booktabs}
\usepackage{bookmark}
\usepackage{array}
\usepackage{caption}
\usepackage{amsthm}
\usepackage{abstract}
\usepackage{booktabs}
\usepackage{cite}
\usepackage{indentfirst}
\usepackage{algorithm}
\usepackage{algorithmic}
\usepackage{hyperref}
\usepackage{subfigure}
\usepackage{multirow}
\usepackage{mathtools}
\usepackage{color}
\newtheorem{mythm}{Theorem}[section]
\newtheorem{mydef}{Definition}[section]
\newtheorem{mylem}{Lemma}[section]
\newtheorem{mycor}{Corollary }[section]
\newtheorem{myrek}{Remark}[section]

\numberwithin{equation}{section}
\makeatletter
\@addtoreset{equation}{section}
\makeatother

\title{\bf A Randomized Singular Value Decomposition for  Third-Order Oriented Tensors}
\author{
\footnote{Published in Journal of Optimization Theory and Applications.}
	Minghui Ding\footnote{School of Mathematical Sciences, Ocean University of China, Qingdao 266100, China.
		E-Mail: {\tt dingminghui@stu.ouc.edu.cn}} \and
    	Yimin Wei\footnote{School of Mathematical Sciences and and Key Laboratory of Mathematics for Nonlinear Sciences, Fudan University, Shanghai 200433, China.
		E-Mail: {\tt ymwei@fudan.edu.cn}}\and
	    Pengpeng Xie\footnote{Corresponding author (P. Xie). School of Mathematical Sciences, Ocean University of China, Qingdao 266100, China.
		E-Mail: {\tt xie@ouc.edu.cn}.}
}
\date{}

\begin{document}	
	\maketitle
\begin{abstract}
 {\bf Abstract}

The oriented singular value decomposition (O-SVD) proposed by Zeng and Ng 
provides a hybrid approach to the t-product based third-order tensor singular value decomposition with the transformation matrix being a factor matrix of the higher order singular value decomposition. Continuing along this vein, this paper explores realizing the O-SVD efficiently by drawing a connection to the tensor-train rank-1 decomposition and gives a truncated O-SVD. Motivated by the success of probabilistic algorithms, we develop a randomized version of the O-SVD and present its detailed error analysis. The new algorithm has advantages in efficiency while keeping good accuracy compared with the current tensor decompositions. Our claims are supported by numerical experiments on several oriented tensors from real applications.
\\ \hspace*{\fill} 
\\ \hspace*{\fill} 
\end{abstract}

\newpage

\section{Introduction}
\hskip 2em Tensors \cite{Kolda2009, Qi2017SIAM, Qi2018Springer} are multidimensional arrays that have been used in diverse fields of applications, including psychometrics \cite{1970Analysis}, image/video and signal processing \cite{2015Tensor}, machine learning \cite{Johan2014}, and web link analysis \cite{Kenny2005}. The compression, sort, analysis, and many other processing of tensor data rely on the tensor decomposition. Various tensor decompositions under different tensor products such as the CANDECOMP/PARAFAC \cite{1970Analysis, 1969Foundations}, higher order singular value decomposition (HOSVD) \cite{1966Some, De2000}, T-SVD \cite{Kilmer2011, Kilmer2013}, T-CUR \cite{CheJOTA2022},  tensor-train \cite{Oseledets2011} and  tensor-train rank-1 (TTr1) SVD (TTr1SVD) \cite{constructiveBatselier2015} have been investigated to extend linear algebra methods to the multilinear context. Among these decompositions, the HOSVD algorithm is not orientation dependent and can achieve high compression ratios if the target rank for the Tucker approximation is small compared to the original dimensions. By contrast, since the Fourier matrix is independent of the tensor, the T-SVD cannot embody the data feature and is not suitable for all orientation-dependent data.  However, we usually confront orientation-dependent tensors which have high correlation among frontal slices in applications. Recently, Zeng and Ng \cite{Zeng2020} proposed  a new decomposition for third-order tensors based on the HOSVD and T-SVD, which is named oriented singular value decomposition (O-SVD). Like the T-SVD, the O-SVD also aims at tensors with fixed orientations and has been demonstrated to be useful in approximation and data compression.

\hskip 2em  While the O-SVD combines the ideas of the HOSVD and T-SVD, it can still be expressed with the sum of outer product terms. This property is also reminiscent of the TTr1SVD which decomposes an arbitrary tensor into a finite sum of orthogonal rank-1 outer products. Unlike the O-SVD, one needs to progressively reshape and compute the SVD of each singular vector to produce the TTr1 decomposition \cite{constructiveBatselier2015}. Therefore, in this paper, we first consider the acceleration of the computation of the O-SVD by means of the constructive approach for the TTr1SVD. Then we turn to the numerical approximation of the O-SVD. A straightforward $r$-term approximation of the O-SVD has been proposed by keeping $r$ terms for which singular values are the largest in magnitude and discarding the other terms. We propose an alternative truncation strategy  for better preserving the original structure inherited from the HOSVD and T-SVD.
	
\hskip 2em In recent years, randomized matrix methods have been used to efficiently and accurately compute approximate low-rank matrix decompositions and the least squares problem (see \cite{Halko2011, boyd2011randomized, DRINEAS2016RandNLA, SHABAT2018246,Wei2016,Xie2019}). These algorithms are easy to implement, and have been extended to the singular value decomposition of tensors based on different tensor products \cite{che2019, CheSIMAX2020, CheJCAM2021, CheJSC2021, 2020Randomized, zhang2018}. Specially, Zhang \emph{et al}. \cite{zhang2018} proposed an algorithm that extends a well-known randomized matrix method to the T-SVD that was called  the RT-SVD, which is more computationally efficient on large data sets. Che and Wei \cite{che2019} designed randomized algorithms for computing the  Tucker and tensor train approximations of tensors with unknown multilinear rank and analyzed their probabilistic error bounds under certain assumptions. Minster \emph{et al}. \cite{2020Randomized} presented  randomized  algorithms of the  HOSVD (RHOSVD) and sequentially truncated HOSVD (RSTHOSVD) in the Tucker representation and gave a detailed probabilistic error analysis for both algorithms. They also applied the adaptive randomized algorithm to find a low-rank representation satisfying a given tolerance and proposed a structure-preserving decomposition where the core tensor retains favorable properties of the original tensor. However, the randomized algorithms mentioned above are not useful for tensors with a fixed orientation involving time series or other ordered data that are highly correlated among slices. Hence, we are motivated to design a new type of randomization strategy corresponding to the O-SVD, and hopefully, this technique can greatly reduce the computational cost while maintaining the accuracy.

\hskip 2em The rest of this paper is organized as follows. In Section \ref{Preliminaries}, we introduce some basic definitions and preliminaries. In Section \ref{OSVDsection},  we briefly introduce the O-SVD, discuss its connection to TTr1SVD and present a new truncation strategy for the O-SVD. Thereafter, in Section \ref{ROSVDsection}, a randomized tensor algorithm based on the O-SVD is proposed. We also give an expected error bound and compare the computational and memory cost with the RT-SVD and RHOSVD. Section \ref{Numericalsection} presents numerical results on the approximation error and the computational complexity of the algorithm, and compares it with some existing methods. Some conclusions are presented in Section \ref{Conclusion}.		
	
\section{Preliminaries}\label{Preliminaries}
\hskip 2em In this section, we introduce definitions and notation used throughout the paper. Scalars are denoted by lowercase letters, e.g. $a$, vectors are denoted by bold-face lowercase letters, e.g. $\boldsymbol{a}$, matrices are denoted by bold-face capitals, e.g. $\boldsymbol{A}$, and tensors are written as calligraphic letters, e.g. $\mathcal{A}$. The $i$th entry of a vector  $\boldsymbol{a}$ is denoted by $a_{i}$, and the $(i,j,k)$th element of a third-order tensor $\mathcal{A}$ is denoted by $a_{ijk}$. For convenience, we sometimes use the MATLAB notation to denote subportions of a matrix or tensor, (e.g., $\mathcal{A}(:,:,k)$ denotes the $k$th frontal slice of a tensor and $\boldsymbol{A}(i,:)$ the $i$th row of a matrix). 

\hskip 2em A mode-$n$ fiber is a column vector defined by fixing every index but the $n$th index, and a mode-$(m,n)$ slice is a matrix defined by fixing every index but the $m$th index and the $n$th index.  A mode-$(1,2)$ slice is also called a frontal slice.  The mode-$n$ unfolding of a tensor $\mathcal{A}\in \mathbb{R}^{I_{1} \times I_{2} \times I_{3} }$ is denoted by $\boldsymbol{A}_{(n)}$ and arranges the mode-$n$ fibers to be the columns of the resulting matrix. 

	\begin{mydef}[ {M}ode-$n$ product \cite{Kolda2009}]\label{d2}
	The mode-$n $ product  of a tensor $\mathcal{A} \in \mathbb{R}^{I_{1} \times I_{2} \times \cdots \times I_{N}}$ by a matrix $\boldsymbol{B} \in \mathbb{R}^{J_{n} \times I_{n}}$, denoted by $\mathcal{A} \times_{n} \boldsymbol{B}${\color{blue},} is a tensor $\mathcal{C} \in \mathbb{R}^{I_{1} \times \cdots \times I_{n-1} \times J_{n} \times I_{n+1} \times \cdots \times I_{N}}$ with
	\begin{equation*}
	c_{i_{1} \ldots i_{n-1} j i_{n+1} \ldots i_{N}}=\sum_{i_{n}=1}^{I_{n}} a_{i_{1} \ldots i_{n-1} i_{n} i_{n+1} \ldots i_{N}} b_{j i_{n}},
	\end{equation*}
	where $n=1,2, \ldots, N$.
\end{mydef}
\begin{mydef}[ {I}nner product \cite{Kolda2009}]
	The inner product of two tensors $\mathcal{A}$, $\mathcal{B} \in \mathbb{C}^{I_{1} \times I_{2} \times \cdots \times I_{N}}$ is defined as
	\begin{equation*}
	\langle\mathcal{A}, \mathcal{B}\rangle=\sum_{i_{1}, i_{2}, \ldots, i_{N}} \overline{a}_{i_{1} i_{2} \cdots i_{N}} b_{i_{1} i_{2} \ldots i_{N}} .
	\end{equation*}
\end{mydef}
 We call two tensors  orthogonal  if their inner product is $0$. The norm of a tensor is taken to be the Frobenius norm $\|\mathcal{A}\|_{F}=\langle\mathcal{A}, \mathcal{A}\rangle^{1 / 2}$.
\begin{mydef}[ {O}uter product \cite{constructiveBatselier2015}]
A third-order rank-1 tensor $\mathcal{A}$  can always be written as the outer product
	\begin{equation*}
	\sigma(\boldsymbol{a} \circ \boldsymbol{b} \circ \boldsymbol{c}) \quad \text { with components } a_{ijk}=\sigma a_{i} b_{j} c_{k}
	\end{equation*}
	with $\sigma \in \mathbb{R}$, whereas $\boldsymbol{a}$, $\boldsymbol{b}$, and $\boldsymbol{c}$ are vectors of arbitrary lengths. Using the mode-$n$ multiplication, this outer product can also be written as
	$\sigma_{\times_{1}} \boldsymbol{a}_{\times_{2}} \boldsymbol{b}_{\times_{3}} \boldsymbol{c}$, where $\sigma$ is now regarded as a $1 \times 1 \times 1$ tensor.
\end{mydef}
	\begin{mydef}[ {T}ensor-tensor product\cite{Zeng2020}]\label{d1}
	The three-mode product of $\mathcal{A}\in \mathbb{R}^{I_{1} \times I_{2} \times I_{3} }$, $\mathcal{B}\in \mathbb{R}^{I_{2} \times I_{4} \times I_{3} }$ denoted by $\mathcal{A} *_{3} \mathcal{B}$, is of size $I_{1} \times I_{4} \times I_{3} $, which is given by
	\begin{equation*}
	\left(\mathcal{A} *_{3} \mathcal{B}\right)(:,:, k)=\mathcal{A}(:,:, k) \mathcal{B}(:,:, k), \quad k=1, \ldots, I_{3}.
	\end{equation*}
	\end{mydef}
	\begin{mydef}[ {T}ranspose]
	If $\mathcal{A}$ is an $I_{1} \times I_{2} \times I_{3} $ tensor, then $\mathcal{A}^{T}$ is an $I_{2} \times I_{1} \times I_{3}$ tensor obtained by  transposing each of the frontal slices, i.e., $\mathcal{A}^{T}(:,:,i)=\mathcal{A}(:,:,i)^{T}$, for $i=1,\ldots, I_{3}$.
	\end{mydef}
\hskip 2em We should note that the definition here of a transpose operation for tensors is different from that in \cite{Kilmer2011}. The following lemma introduces some properties of the mode-$n$ product.
      \begin{mylem}[\cite{Zeng2020, Kolda2009} ]\label{gaixie}
	Let $\mathcal{A} \in \mathbb{R}^{I_{1} \times I_{2} \times I_{3} } ,\boldsymbol{M}\in \mathbb{R}^{J_{n} \times I_{n}}$, $\mathcal{B}=\mathcal{A} \times_{n}\boldsymbol{M}, n=1,2,3.$ Then\\
{\rm (1)} $\boldsymbol{B}_{(n)}=\boldsymbol{M}\boldsymbol{A}_{(n)}$;\\
{\rm	(2)} If $\boldsymbol{M}^{(n)}$ is  orthonormal for $J_{n}=I_{n}$, then $ \mathcal{A}=\mathcal{B} \times_{n}\boldsymbol{M}^{T}$;\\
{\rm	(3)} For $\boldsymbol{M}_{1} \in \mathbb{R}^{J_{n} \times I_{n}},\boldsymbol{M}_{2} \in \mathbb{R}^{ J^{'}_{n}\times J_{n}}$, then $\left( \mathcal{A} \times_{n}\boldsymbol{M}_{1}\right) \times_{n}\boldsymbol{M}_{2}=\mathcal{A} \times_{n}\left( \boldsymbol{M}_{2}\boldsymbol{M}_{1}\right)$.
	\end{mylem}
	\begin{mydef}[ {T}ensor rank\cite{De2000, Kilmer2013}]
	Let $\mathcal{A} \in \mathbb{R}^{I_{1} \times I_{2} \times I_{3} }$.\\
{\rm	(1)} The $n$-rank of $\mathcal{A}$, denoted by $\mathtt{rank}_{n}(\mathcal{A})$, is the dimension of the vector space spanned by all mode-$n$ fibers. For example, $\mathtt{rank}_{3}(\mathcal{A})=\mathtt{rank}(\boldsymbol{A}_{(3)})$.\\
{\rm	(2)} The multirank of  $\mathcal{A}$ is a mode-3 fiber  $\mathbf{rank}_{m}(\mathcal{A})\in \mathbb{R}^{ I_{3}}$ such that $\mathbf{rank}_{m}(\mathcal{A})(i)$ is the rank of  $\left( \mathcal{A} \times_{3} \boldsymbol{M}\right) (:,:,i)$ where $\boldsymbol{M}$  represents different meanings under different tensor products.
	\end{mydef}
The Tucker rank of tensor $\mathcal{A}$ is a vector with  its elements being the ranks of matrix unfoldings with
respect to the corresponding modes, i.e., $(\mathtt{rank}_{1}(\mathcal{A}), \mathtt{rank}_{2}(\mathcal{A}), \mathtt{rank}_{3}(\mathcal{A})).$
If $\boldsymbol{M}$ is the left singular matrix of $\boldsymbol{A}_{(3)}$ with $\boldsymbol{A}_{(3)}=\boldsymbol{U}^{(3)}\boldsymbol{S}^{(3)}\boldsymbol{V}^{(3)T},$  it is shown in \cite{Zeng2020} that
$$ \max{\mathbf{rank}_{m}(\mathcal{A})}=\max{\{\mathtt{rank}(( \mathcal{A} \times_{3} \boldsymbol{U}^{(3)T}) (:,:,i))}\} \leq \min{\{\mathtt{rank}_{1}(\mathcal{A}), \mathtt{rank}_{2}(\mathcal{A})}\}.$$

\section{O-SVD} \label{OSVDsection}
\hskip 2em We first review the O-SVD developed by Zeng and Ng \cite{Zeng2020}, which is built on the operations of tensors introduced in Section \ref{Preliminaries}. In order to better understand the O-SVD and improve its numerical realization, we show that the outer product form of the O-SVD is in fact the TTr1SVD introduced in \cite{constructiveBatselier2015}. We then proceed to a type of truncated O-SVD (TO-SVD) in an analogous manner to the  truncated T-SVD (TT-SVD) \cite{Kilmer2011} and develop a rigorous error analysis.

      \begin{mythm}[O-SVD\cite{Zeng2020}]\label{OSVD}
	Let $\mathcal{A} \in \mathbb{R}^{I_{1} \times I_{2} \times I_{3}}$ and  $R_{3}=\mathtt{rank}_{3}(\mathcal{A})$. There exists an orthogonal  matrix $\boldsymbol{U}^{(3)} \in \mathbb{R}^{I_{3} \times I_{3}}$, three tensors $\mathcal{U} \in \mathbb{R}^{I_{1} \times I_{1} \times  I_{3}}$, $\mathcal{S}\in \mathbb{R}^{I_{1} \times I_{2} \times  I_{3}}$, $\mathcal{V} \in \mathbb{R}^{I_{2} \times I_{2} \times  I_{3}}$ such that
	\begin{equation}\label{A}
	\mathcal{A}=\left(\mathcal{U} *_{3} \mathcal{S} *_{3} \mathcal{V}\right) \times_{3} \boldsymbol{U}^{(3)},
	\end{equation}
	where\\
{\rm 	(1)} $\mathcal{U}(:,:, i)$, $\mathcal{V}(:,:, i)$ are orthogonal and~$\mathcal{S}(:,:, i)$ is a nonnegative diagonal matrix for $i=1,2, \ldots, R_{3}$;\\
{\rm 	(2)} $\mathcal{U}(:,:, i)$, $\mathcal{V}(:,:, i)$ and $\mathcal{S}(:,:, i)$ are all zero matrices for $i=R_{3}+1, \ldots, I_{3}$.
	\end{mythm}
The diagonal elements $ s_{jji}$ of  each  frontal slice of  $\mathcal{S}$ are called the singular values of the pair $(\mathcal{A},\mathcal{S})$.
	\begin{mythm}[\cite{Zeng2020} ]\label{dengyu}
	Let the core tensor corresponding to the O-SVD of $\mathcal{A} \in \mathbb{R}^{I_{1} \times I_{2} \times I_{3}}$ be $\mathcal{S}$. Then each slice $\mathcal{S}(:,:, i) \in \mathbb{R}^{I_{1} \times I_{2}}$ has the following property:\\
	$$\left\|\mathcal{S}(:,:, 1)\right\|_{F} \geq\left\|\mathcal{S}(:,:, 2)\right\|_{F} \geq \cdots \geq\left\|S\left(:,:, I_{3}\right)\right\|_{F} \geq 0,$$
	where $\left\|\mathcal{S}(:,:, i)\right\|_{F}=\sigma_{i} $ with $\sigma_{i}$ being the $i$th largest singular value of $\boldsymbol{A}_{(3)}$ and	
	$$\sigma_{i} \geq \mathcal{S}(1,1, i) \geq \mathcal{S}(2,2, i) \geq \cdots \geq \mathcal{S}(r_{2}, r_{2}, i) \geq 0 $$	
	for $i=1, 2,\ldots, r_{1}$, where $r_{1}=\min \left\lbrace I_{3}, I_{1} I_{2}\right\rbrace $, $r_{2}=\min \left\{I_{1}, I_{2}\right\}$.
	\end{mythm}

\subsection{O-SVD and TTr1SVD}\label{gaijin}
\hskip 2em It follows from Theorem \ref{OSVD} that there are two steps in the computational procedure of the O-SVD. The first step is to find a basis of the space spanned by the frontal slices of $\mathcal{A}$. Specifically, one needs to conduct the “economical” SVD of the $I_{3} \times I_{1} I_{2}$ matrix $\boldsymbol{A}_{(3)}$
\begin{equation*}\label{di1bu}
\boldsymbol{A}_{(3)}=\boldsymbol{U}^{(3)}\boldsymbol{S}^{(3)}\boldsymbol{V}^{(3)T},
\end{equation*}
where  the number of non-zero singular values obtained is equal to the number of  the desired basis. Then, the frontal slices of $\widetilde{\mathcal{A}}= \mathcal{A} \times_{3}\boldsymbol{U}^{(3)T}$ are the basis we are looking for.  Motivated by the observation that
\begin{equation}\label{di2bu}
\widetilde{\mathcal{A}}(:,:,i)=\sigma_{i}\widetilde{\boldsymbol{V}}\mathclap{_{i}}\,=\sigma_{i}\mathtt{reshape}(\boldsymbol{V}^{(3)}(:,i),\left[I_{1},I_{2} \right] ),
\end{equation}
where the operator $\mathtt{reshape}$  returns the $I_{1}$-by-$I_{2}$ matrix $\widetilde{\boldsymbol{V}}\mathclap{_{i}}$ whose elements are taken columnwise from $\boldsymbol{V}^{(3)}(:,i)$, we can compute the SVD directly for each matrix $\widetilde{\boldsymbol{V}}\mathclap{_{i}}$
\begin{equation}\label{di3bu}
\widetilde{\boldsymbol{V}}\mathclap{_{i}}=\boldsymbol{U}\mathclap{_{i}}\,\boldsymbol{S}_{i}\boldsymbol{V}\mathclap{_{i}}^{T}
\end{equation}
instead of forming $\widetilde{\mathcal{A}}$ first. This procedure is directly inspired by the algorithm of TTr1 decomposition \cite{constructiveBatselier2015}, which requires recursively reshaping the right singular vectors $\boldsymbol{V}^{(3)}(:,i)$, and computing their SVDs. This algorithm is called TTr1SVD and gives rise to the formation of a tree. Since the first step of the TTr1SVD algorithm is to expand the tensor along the selected mode, the O-SVD is in fact the  TTr1SVD for the third-order oriented tensors with the processing order $\rho = [3, 1, 2]$. Let $\sigma_{ij}$ denote the $j$th largest singular value of $\widetilde{\boldsymbol{V}}\mathclap{_{i}}$  where $i=1,2,\ldots, r_{1} $,  $j=1,2,\ldots, r_{2}$. Substituting (\ref{di3bu}) into  (\ref{di2bu}), it is easy to derive that
 \begin{equation}\label{chengji}
 s_{jji}=\sigma_{i}\sigma_{ij}, \, \mathcal{U}(:,:, i)=\boldsymbol{U}\mathclap{_{i}}\,, \mathcal{V}(:,:, i)=\boldsymbol{V}\mathclap{_{i}}^{T}.
 \end{equation}
 Since any matrix can be written  as a sum of rank-$1$ terms, we can also rewrite $\mathcal{A}$ as
 \begin{equation}\label{zhiyi}
 \mathcal{A}=\sum_{i=1}^{r_{1}}\sum_{j=1}^{r_{2}} s_{jji} \times_{1} \boldsymbol{u}_{ij}\times_{2} \boldsymbol{v}_{ij} \times_{3}\boldsymbol{u}_{i},
 \end{equation}
 where  $\boldsymbol{u}_{i}$, $\boldsymbol{u}_{ij}$, $\boldsymbol{v}_{ij}$ are the column vectors of $\boldsymbol{U}^{(3)}$, $\boldsymbol{U}\mathclap{_{i}}\,$, $\boldsymbol{V}\mathclap{_{i}}$ respectively. Consequently, the O-SVD  also has three main features that render it similar to the matrix SVD.
\begin{mycor}[\cite{constructiveBatselier2015}]
Let (\ref{zhiyi}) be the outer product form of the O-SVD of  $\mathcal{A} \in \mathbb{R}^{I_{1} \times I_{2} \times I_{3}}$, and the number of rank-1 terms  be  $R=r_{1}r_{2}$. Then,\\
{\rm (1)} the  scalars $ s_{jji}$ are the weights of the outer products in the decomposition,\\
{\rm (2)} the  outer products affiliated with each singular value are tensors of unit Frobenius norm, since each product vector (or mode vector) is a unit vector, and  \\
{\rm (3)} each  outer product in the decomposition is orthogonal to all the others.
\end{mycor}

\hskip 2em Furthermore, we can obtain a more economical expression, similar to the form of the  $\left( L_{r}, L_{r}, 1\right)$-term decomposition \cite{optimizationsorber2013}.
\begin{mycor}\label{block}
	Let (\ref{A}) be the O-SVD of  $\mathcal{A} \in \mathbb{R}^{I_{1} \times I_{2} \times I_{3}}$. Then\\
	\begin{equation}
\mathcal{A}=\sum_{r=1}^{R_{3}} \boldsymbol{H}\mathclap{_{r}} \,\text{\tiny{$\otimes$}} \boldsymbol{u}_{r}, \quad \mathtt{rank}\left(\boldsymbol{H}\mathclap{_{r}}\,\right)=L_{r}>0 \, \text { for } 1 \leq r \leq R_{3},	
	\end{equation}
 where $\boldsymbol{H}\mathclap{_{r}}=\widetilde{\mathcal{A}}(:,:,r)=\mathcal{U}(:,:,r) \mathcal{S}(:,:,r)\mathcal{V}(:,:,r)$,  $\boldsymbol{u}_{r}=\boldsymbol{U}^{(3)}(:,r)$, and $\text{\tiny{$\otimes$}}$  is the tensor product defined by $(H \text{\tiny{$\otimes$}} \boldsymbol{u})(i, j, k)= h_{ij}u_{k}$.
\end{mycor}

\subsection{TO-SVD and  its  Error  Analysis}
  \hskip 2em For an orientation-dependent tensor $\mathcal{A} \in \mathbb{R}^{I_{1} \times I_{2} \times I_{3}}$, the numerical rank of $\boldsymbol{A}_{(3)}$ is usually much smaller than $I_3$.
  In this circumstance, we can utilize the truncation strategy to efficiently compute an approximate O-SVD.
  The Eckart-Young theorem \cite{Eckart1936} states that an optimal rank-$k$ approximation to a matrix  can be constructed using the rank-$k$ truncated SVD.
  Similarly, an $r$-term approximation for the O-SVD can be obtained by truncating (\ref{zhiyi}) to the first $r$ terms.
\begin{mylem}[Approximation \cite{Zeng2020,constructiveBatselier2015}]\label{Rxiang}
Let $\mathcal{A} \in \mathbb{R}^{I_{1} \times I_{2} \times I_{3}}$ and denote $\widetilde{\sigma_{i}}$ as the $i$th singular value of the pair $(\mathcal{A},\mathcal{S})$ in descending order. Denote by $\mathcal{A}_{r}$ the $r$-term approximation by the O-SVD. Then we have
\begin{equation}
\left\|\mathcal{A}-\mathcal{A}_{r}\right\|_{F}^{2}=\sum_{i=r+1}^{r_{1}r_{2}} \widetilde{\sigma_{i}}^{2}.
\end{equation}
 \end{mylem}
Notice that the $r$-term approximation of $\mathcal{A}$ requires reordering all the singular values, finding the outer product corresponding to each singular value and adding them one by one, which is undoubtedly laborious and time-consuming. At the same time, the original structures of Theorem \ref{OSVD} and Corollary \ref{block} cannot be maintained.

\hskip 2em To overcome this drawback, we adopt an alternative truncation strategy.
 Recall that the factor matrix of the truncated HOSVD (THOSVD)\cite{newvannieuwenhoven2012} is obtained from a truncated SVD of the mode-$k$ unfolding of the tensor. For the TT-SVD, it consists of transforming the tensor to the Fourier domain and applying the truncated SVD to each frontal slice of the tensor. Following the procedure of the O-SVD, we consider a truncation method combining the ideas of TT-SVD and THOSVD.
We first perform a truncated SVD of $\boldsymbol{A}_{(3)}$ to get the approximate matrix $\boldsymbol{U}_{k_{1}}^{(3)}$ of the left singular matrix $\boldsymbol{U}^{(3)}$, where $k_1$ is the target truncation rank. Secondly, for each frontal slice of $\mathcal{A} \times_{3}(\boldsymbol{U}_{k_1}^{(3)})^{T}$, we conduct the economical SVD with different target truncation terms. Let $\boldsymbol{k}_{2}=\left[k_{21}, \ldots, k_{2k_{1}}\right] ^{T}$ be the target multirank of the second step. Obviously, this kind of truncation leads to different nonzero blocks in each frontal slice of $\mathcal{S}$, and so do $\mathcal{U}$ and $\mathcal{V}$. For the convenience of description, we set $k_2 = \max\{k_{21}, \ldots, k_{2k_{1}}\}$. Now we are ready to summarize the above discussion in the following definition.

	\begin{mydef}[$\boldsymbol{k}$-term TO-SVD]
	Given a tensor $\mathcal{A} \in \mathbb{R}^{I_{1} \times I_{2} \times I_{3}}$, define the truncation of the O-SVD to $\boldsymbol{k}$ terms of $\mathcal{A}$ as
	\begin{equation}\label{AK}
	\mathcal{A}_{\boldsymbol{k}}=\left( \mathcal{U}_{k_{2}} *_{3}\mathcal{S}_{k_{2}}*_{3}\mathcal{V}_{k_{2}}\right)  \times_{3} \boldsymbol{U}_{k_{1}}^{(3)},
	\end{equation}
	where \\
	{\rm (1)} $\boldsymbol{k}=\left[ k_{1}; \boldsymbol{k}_{2}\right]\in \mathbb{R}^{k_{1}+1 }$ is the target rank vector; \\
	{\rm (2)}  $\mathcal{U}_{k_{2}}(:,:,i) \in \mathbb{R}^{I_{1} \times k_{2}}$ and $\mathcal{V}^{T}_{k_{2}}(:,:,i)\in \mathbb{R}^{I_{2} \times k_{2}}$ have $k_{2i}$ orthogonal columns for $i=1,2,\ldots,k_{1}$. $\mathcal{S}_{k_{2}}(:,:,i) \in \mathbb{R}^{k_{2} \times k_{2}}$ is a nonnegative diagonal matrix for $i=1,2,\ldots,k_{1}$;\\
	{\rm (3)} $\boldsymbol{U}_{k_{1}}^{(3)}=\boldsymbol{U}^{(3)}(:,1:k_{1}) \in \mathbb{R}^{I_{3} \times k_{1}}$.
	\end{mydef}

	\hskip 2em  In Algorithm \ref{TO-SVD}, we show how the TO-SVD can be implemented in combination with the improved methods mentioned in Subsection \ref{gaijin}. The error of the TO-SVD is presented in Theorem \ref{best}.

	\begin{algorithm}[htb]
	\caption{$\boldsymbol{k}$-term TO-SVD}
	\label{TO-SVD}
	\begin{algorithmic}[1] 
	\REQUIRE $\mathcal{A}\in \mathbb{R}^{I_{1} \times I_{2} \times I_{3} }$, target truncation vector $\boldsymbol{k}=\left[ k_{1}; \boldsymbol{k}_{2} \right]$, $\boldsymbol{k}_{2}=\left[k_{21}, \ldots, k_{2k_{1}}\right] ^{T}$,
	\ENSURE	$\boldsymbol{U}_{k_{1}}^{(3)} \in \mathbb{R}^{I_{3} \times k_{1}}$, $\mathcal{U}_{k_{2}}\in \mathbb{R}^{I_{1} \times k_{2} \times k_{1} }$, $\mathcal{S}_{k_{2}}\in \mathbb{R}^{k_{2} \times k_{2} \times k_{1} }$, $\mathcal{V}_{k_{2}}\in \mathbb{R}^{k_{2} \times I_{2} \times k_{1} }$
	\STATE Initialization: $\mathcal{U}_{k_{2}}$, $\mathcal{S}_{k_{2}}$, $\mathcal{V}_{k_{2}}$ are zero tensors of appropriate size;
	\STATE$\left[ \boldsymbol{U}_{k_{1}}^{(3)},\boldsymbol{S}_{k_{1}}^{(3)},\boldsymbol{V}\mathclap{_{k_{1}}}^{(3)}\right]=\mathrm{svds}\left( \boldsymbol{A}_{(3)}, k_{1}\right)$;
	\STATE for $i=1,2,\ldots ,k_{1}$, do
	\STATE \quad $\widetilde{\boldsymbol{V}}\mathclap{_{i}}=\mathtt{reshape}(\boldsymbol{V}\mathclap{_{k_{1}}}^{(3)}(:,i),\left[I_{1},I_{2} \right] )$;
	\STATE\quad $\left[\boldsymbol{U}\mathclap{_{i}}\,,\boldsymbol{S}_{i},\boldsymbol{V}\mathclap{_{i}}\,\right]=\mathrm{svds}(\widetilde{\boldsymbol{V}}\mathclap{_{i}}\,,k_{2i})$;
	\STATE\quad $\mathcal{U}_{k_{2}}(:,1:k_{2i},i)=\boldsymbol{U}\mathclap{_{i}}\,, ~\mathcal{S}_{k_{2}}(1:k_{2i},1:k_{2i},i)=\boldsymbol{S}_{k_{1}}^{(3)}(i,i)\boldsymbol{S}_{i}, ~\mathcal{V}_{k_{2}}(1:k_{2i},:,i)=\boldsymbol{V}\mathclap{_{i}}^{T}$;
	\STATE end for	
	\end{algorithmic}
	\end{algorithm}
	\begin{mythm}\label{best}
	Let $\mathcal{A}_{\boldsymbol{k}}$ be  the truncation of the O-SVD to $\boldsymbol{k}$ terms of  $\mathcal{A}  \in \mathbb{R}^{I_{1} \times I_{2} \times I_{3}}$. Then
	\begin{equation}\label{budengshi}
	\left\|\mathcal{A}- \mathcal{A}_{\boldsymbol{k}}\right\|^{2}_{F} =\sum_{i=1}^{k_{1}}\sum_{j=k_{2i}+1}^{r_{2}}s_{jji}^{2}+\sum_{i=k_{1}+1}^{I_{3}}\sum_{j=1}^{r_{2}} s_{jji}^{2}.
	\end{equation}
	\end{mythm}
	\begin{proof}
	Since the Frobenius norm is unitarily invariant, by (\ref{A}) and (\ref{AK}), we have
	\begin{equation*}
	\begin{aligned}
	\left\|\mathcal{A}- \mathcal{A}_{\boldsymbol{k}}\right\|_{F}&=\left\|\mathcal{A}\times_{3}\boldsymbol{U}^{(3)T}- \mathcal{A}_{\boldsymbol{k}}\times_{3}\boldsymbol{U}^{(3)T}\right\|_{F}\\
	&=\left\|\mathcal{U} *_{3} \mathcal{S} *_{3} \mathcal{V}-\left( \mathcal{U}_{k_{2}} *_{3}\mathcal{S}_{k_{2}}*_{3}\mathcal{V}_{k_{2}}\right) \times_{3}\left( \boldsymbol{U}^{(3)T}\boldsymbol{U}_{k_{1}}^{(3)}\right) \right\|_{F}.\\
	\end{aligned}
	\end{equation*}
	Noting that  $ \boldsymbol{U}^{(3)T}\boldsymbol{U}_{k_{1}}^{(3)}=\left[ \begin{array}{c}  \boldsymbol{I}_{k_{1}}\\\boldsymbol{0} \end{array}\right]$, now let $\widetilde{ \mathcal{U}}_{k_{2}}\left(:,:,i \right)=\mathcal{U}_{k_{2}}\left(:,:,i \right)$, $\widetilde{ \mathcal{S}}_{k_{2}}\left(:,:,i \right)=\mathcal{S}_{k_{2}}\left(:,:,i \right)$, $\widetilde{ \mathcal{V}}_{k_{2}}\left(:,:,i \right)=\mathcal{V}_{k_{2}}\left(:,:,i \right)$, for $i=1, 2,\ldots, k_{1}$ and $\widetilde{ \mathcal{U}}_{k_{2}}\left(:,:,i \right)=\boldsymbol{0}$, $\widetilde{ \mathcal{S}}_{k_{2}}\left(:,:,i \right)=\boldsymbol{0}$, $\widetilde{ \mathcal{V}}_{k_{2}}\left(:,:,i \right)=\boldsymbol{0}$  of the corresponding dimension for $i=k_{1}+1, \ldots,I_{3}$, from which we can obtain 	
	\begin{equation*}
	\begin{aligned}
	\left\|\mathcal{A}- \mathcal{A}_{\boldsymbol{k}}\right\|^{2}_{F}&=\left\|\mathcal{U} *_{3} \mathcal{S} *_{3} \mathcal{V}-\widetilde{ \mathcal{U}}_{k_{2}} *_{3}\widetilde{\mathcal{S}}_{k_{2}}*_{3}\widetilde{\mathcal{V}}_{k_{2}} \right\|^{2}_{F}\\
	&=\sum_{i=1}^{I_{3}}\left\| \mathcal{U}(:,:,i) \mathcal{S}(:,:,i)\mathcal{V}(:,:,i)-\widetilde{\mathcal{U}}_{k_{2}}(:,:,i) \widetilde{\mathcal{S}}_{k_{2}}(:,:,i)\widetilde{\mathcal{V}}_{k_{2}}(:,:,i) \right\|^{2}_{F} \\
	&= \sum_{i=1}^{I_{3}}\left\| \mathcal{S}(:,:,i)-\left[ \begin{array}{cc} \widetilde{\mathcal{S} }_{k_{2}}(1:k_{2i},1:k_{2i},i)& \\&\boldsymbol{0} \end{array}\right] \right\|^{2}_{F} \\
	&=\sum_{i=1}^{k_{1}}\sum_{j=k_{2i}+1}^{r_{2}}s_{jji}^{2}+\sum_{i=k_{1}+1}^{I_{3}}\sum_{j=1}^{r_{2}}s_{jji}^{2},
	\end{aligned}
	\end{equation*}
	where we used $\mathcal{U}(:,:,i)^{T}\widetilde{\mathcal{U}}_{k_{2}}(:,:,i)=\left[ \begin{array}{c}  \boldsymbol{I}_{k_{2i}}\\\boldsymbol{0} \end{array}\right]$ and $\widetilde{\mathcal{V}}_{k_{2}}(:,:,i)\mathcal{V}(:,:,i)^{T}=\left[ \begin{array}{cc}  \boldsymbol{I}_{k_{2i}} &\boldsymbol{0} \end{array}\right]$ in the third equality. Finally, observing that $\widetilde{\mathcal{S} }_{k_{2}}(1:k_{2i},1:k_{2i},i)= \mathcal{S}(1:k_{2i},1:k_{2i},i)$ for $i=1,2,\ldots,k_{1}$ and $s_{jji}$ is  the $j$th singular value of $\mathcal{S}(:,:,i) $, we get the desired result.
	\end{proof}

\hskip 2em The error will serve as an important reference  to compare the accuracy of  the randomized O-SVD with the deterministic one.
Comparing Theorem \ref{best} with Lemma \ref{Rxiang}, we have
$$
\left\|\mathcal{A}-\mathcal{A}_{k}\right\|_{F} \geq\left\|\mathcal{A}-\mathcal{A}_{r}\right\|_{F},
$$
where $r=\sum_{i=1}^{k_{1}} k_{2i}$. However, the TO-SVD is cheaper to compute and retains the original structure of  the O-SVD.
\section{Randomized O-SVD} \label{ROSVDsection}
\hskip 2em  Randomized algorithms play a key role in low-rank approximations of large matrices. In this section, the scheme of the matrix randomized SVD (R-SVD) \cite[Section 4]{Halko2011} is extended to a randomized O-SVD algorithm (RO-SVD). We first review  the matrix R-SVD method and its expected error estimate that we will use later to analyze the error in the RO-SVD method. We also discuss the computational and memory costs of the proposed randomized algorithm.
\subsection{Randomized SVD}
\hskip 2em Randomized SVD, popularized by \cite{Halko2011}, is a computationally efficient way to compute a low-rank approximation of a matrix. Given a matrix  $\boldsymbol{A}\in \mathbb{R}^{m \times n}(m\leq n)$, a target rank $k$, and an oversampling parameter $p$, we first multiply the matrix $\boldsymbol{A}$ by a Gaussian random  matrix $\boldsymbol{\Omega} \in \mathbb{R}^{n \times (k+p)}$. The matrix $\boldsymbol{Y=A\Omega}$ thus contains random linear combinations of the columns of $\boldsymbol{A}$. A thin QR of $\boldsymbol{Y}$ is then computed, so that $\mathrm{range}(\boldsymbol{Y})=\mathrm{range}(\boldsymbol{Q})$. The idea is if $\boldsymbol{A}$ has rapidly decaying singular values, the dominant part of the range of $\boldsymbol{A}$ is marked by the first $k$ or so the left singular vectors, that is, $\boldsymbol{A} \approx \boldsymbol{QQ}^{T}\boldsymbol{A}=\widehat{\boldsymbol{A}}$. Then, we compute a thin SVD of much smaller matrix $\boldsymbol{Q}^{T}\boldsymbol{A}=\boldsymbol{U}\boldsymbol{S}_{k}\boldsymbol{V}\mathclap{_{k}}^{T}$, truncate down to the target rank $k$, and compute $\boldsymbol{U}\mathclap{_{k}}=\boldsymbol{Q}\boldsymbol{U}$ to obtain the low-rank approximation $\widehat{\boldsymbol{A}}=\boldsymbol{U}\mathclap{_{k}}\,\boldsymbol{S}_{k}\boldsymbol{V}\mathclap{_{k}}^{T}$.

\hskip 2em The techniques described above work well for matrices whose singular values exhibit some decay, but they may produce a poor basis  when the input matrix has a flat singular spectrum or when the input matrix is very large. A modified scheme originally proposed in \cite{Rokhlin2009}, makes use of power iteration to improve the accuracy of randomized algorithms in these situations. Specifically, the projection step $\boldsymbol{Y}=\boldsymbol{A}\boldsymbol{\Omega}$ is replaced with $\boldsymbol{Y}=(\boldsymbol{A}\boldsymbol{A}^{T})^q\boldsymbol{A\Omega}$ for small integer $q$, where $q=1$ or $q=2$ usually suffices in practice. In particular, when $q = 0$, the algorithm is equivalent to the basic randomized  SVD \cite{Halko2011}. In our paper,  we adopt Algorithm \ref{alg:r-svd}, a numerically stable version, which is available in \cite[Algorithm 4.4]{Halko2011} that alternates the QR factorization with the matrix{\color{blue}\mbox{-}}matrix products and assume that it can be invoked as $[\boldsymbol{U}\mathclap{_{k}}\,,\boldsymbol{S}_{k}, \boldsymbol{V}\mathclap{_{k}}\,]=\mathrm{rsvd}(\boldsymbol{A}, \boldsymbol{\Omega}, k, p, q)$.
\begin{myrek}
 For Gaussian test matrices, it is adequate to choose the oversampling parameter to be a small constant, such as $p = 5$ or $p = 10$. There is rarely any advantage to select $p>k$. This observation, first presented in \cite{randomizedmartinsson2011}, demonstrates that a Gaussian test matrix results in a negligible amount of extra computation.
\end{myrek}
\begin{algorithm}
	\caption{R-SVD method with power iteration\cite{Halko2011} }
	\label{alg:r-svd}
	\begin{algorithmic}[1] 
		\REQUIRE $A\in \mathbb{R}^{m \times n }$, Gaussian random  matrix $\boldsymbol{\Omega}\in \mathbb{R}^{n \times (k+p)}$, target truncation term $k$, a parameter $q$, and oversampling parameter $p$
		\ENSURE $\boldsymbol{U}\mathclap{_{k}} \in \mathbb{R}^{n \times k}$, $\boldsymbol{S}_{k} \in \mathbb{R}^{k \times k}$, $\boldsymbol{V}\mathclap{_{k}} \in \mathbb{R}^{n \times k}$
		\STATE Form $\boldsymbol{Y}\mathclap{_{0}} = \boldsymbol{A}\boldsymbol{\Omega}$ and compute its QR factorization $\boldsymbol{Y}\mathclap{_{0}} = \boldsymbol{Q}_{0}\boldsymbol{R}_{0}$;
		\STATE for $j=1,2,\ldots ,q$, do
		\STATE\quad Form $\widehat{ \boldsymbol{Y}}\mathclap{_{j}}  = \boldsymbol{A}^{T}\boldsymbol{Q}_{j-1}$ and compute its QR factorization $\widehat{ \boldsymbol{Y}}\mathclap{_{j}} = \widehat{\boldsymbol{Q}}_{j}\widehat{\boldsymbol{R}}_{j}$;
		\STATE\quad Form $\boldsymbol{Y}\mathclap{_{j}} = \boldsymbol{A}\widehat{\boldsymbol{Q}}_{j}$ and compute its QR factorization $\boldsymbol{Y}\mathclap{_{j}} = \boldsymbol{Q}_{j}\boldsymbol{R}_{j}$;
		\STATE end
		\STATE $\boldsymbol{Q}=\boldsymbol{Q}_{q}$;
		\STATE Form $\boldsymbol{B}=\boldsymbol{Q}^{T}\boldsymbol{A} \in \mathbb{R}^{  (k+p)\times n}$;
		\STATE $\left[ \boldsymbol{U}, \boldsymbol{S}_{k}, \boldsymbol{V}\mathclap{_{k}}\, \right] =\mathrm{svds}\left(\boldsymbol{B},k \right) $;
		\STATE Form $\boldsymbol{U}\mathclap{_{k}}=\boldsymbol{Q}\boldsymbol{U}$.
	\end{algorithmic}
\end{algorithm}

  \hskip 2em  When $\boldsymbol{A}$ is dense and of size $n \times n$, the basic randomized  SVD takes $\mathcal{O}(kn^{2})$ flops and Algorithm \ref{alg:r-svd} requires $2q + 1$ times as many matrix-matrix multiplications as the basic randomized  SVD.  An error bound for Algorithm \ref{alg:r-svd} in the Frobenius norm is presented below, which can be found in \cite{zhang2018}.
\begin{mythm}[Average Frobenius  Error for Algorithm \ref{alg:r-svd}]\label{matrix error}
	Let $\boldsymbol{A} \in \mathbb{R}^{m \times n}$ and  $\boldsymbol{\Omega} \in \mathbb{R}^{n \times(k+p)}$ be a Gaussian random matrix with $p \geq 2$ being the oversampling parameter. Suppose
	that $\boldsymbol{Q}$ is obtained from Algorithm \ref{alg:r-svd} and $\boldsymbol{B}\mathclap{_{k}}$ is the rank-$k$ truncated SVD of $\boldsymbol{Q}^{T}\boldsymbol{A}$, then
	\begin{equation}
	\mathbb{E}_{\boldsymbol{\Omega}}\left\|\boldsymbol{A}-\boldsymbol{Q} \boldsymbol{Q}^{T}\boldsymbol{A}\right\|_{F}^{2} \leq \mathbb{E}_{\boldsymbol{\Omega}}\left\|\boldsymbol{A}-\boldsymbol{Q} \boldsymbol{B}_{k}\right\|_{F}^{2} \leq\left(1+\frac{k}{p-1} \tau_{k}^{4 q}\right)\left(\sum_{j>k}^{\min \left\{m,n\right\}} \sigma_{j}^{2}\right),
	\end{equation}
	where $k$ is a target truncation term, $q$ is the number of iterations, $ \sigma_{j}$ is the $j$th singular value of $\boldsymbol{A}$, and $\tau_{k}=\sigma_{k+1} / \sigma_{k} \ll 1$ is the singular value gap.
\end{mythm}
\begin{myrek}
 Instead of $\left\|\boldsymbol{A}-\boldsymbol{Q} \boldsymbol{B}_{k}\right\|_{F}^{2}$, we will use $\left\|\boldsymbol{A}-\boldsymbol{U}\mathclap{_{k}}\,\boldsymbol{U}\mathclap{_{k}}^{T} \boldsymbol{A}\right\|_{F}^{2}$. It is straightforward to show the equivalence between the two forms \cite[Section 5.3]{Saibaba2019}.
\end{myrek}
\subsection{RO-SVD and  its  Error Analysis}
\hskip 2em  The goal of the RO-SVD method is to find a good approximate O-SVD of tensor $\mathcal{A}$ with less storage and time. There are two stages  in producing the approximation, which are summarized in Algorithm \ref{alg:ROSVD}. The basic RO-SVD method is a specific case of Algorithm \ref{alg:ROSVD} that all iteration parameters are chosen to be 0. In the first stage,  setting $k_{1}$ as the first target rank, a full SVD of  $\boldsymbol{A}_{(3)}$ is replaced with a randomized SVD  to find an orthonormal matrix $\boldsymbol{U}_{k_{1}}^{(3)}$ such that
\begin{equation}\label{step1}
\begin{aligned}
\widehat{\mathcal{A}} & = \mathcal{A} \times_{3}\boldsymbol{U}_{k_{1}}^{(3)T}.
\end{aligned}
\end{equation}
 This allows us to express $\mathcal{A} \approx \widehat{\mathcal{A}}\times_{3}\boldsymbol{U}_{k_{1}}^{(3)}$. Then, the second stage is to connect this low 3-rank tensor $\widehat{\mathcal{A}}$ representation to a randomized tensor SVD, where we apply the randomized SVD to each frontal slice of  $\widehat{\mathcal{A}}$ with  different target  truncation term $k_{2i}$ for $i=1,2,\ldots, k_{1}$. This means that the target multirank is a vector $\boldsymbol{k}_{2}$ whose elements are  $k_{2i}$. For the convenience of notation, we further define an iteration vector as $\boldsymbol{q}=\left(q_{1},q_{2},\ldots, q_{k_{1}} \right)^{T}$ with employing different iteration count $q_{i}$ and denote  $\boldsymbol{k}$ = $\left[ k_{1};\,\boldsymbol{k}_{2}\right]$. Thus we find the tensor $\mathcal{U}_{k_{2}}$ such that
\begin{equation}\label{step2}
\widehat{\mathcal{A}}\approx  \mathcal{U}_{k_{2}}*_{3} \mathcal{U}_{k_{2}}^{T}*_{3} \widehat{\mathcal{A}}=\widehat{\mathcal{A}}_{k_{2}}.
\end{equation}
Obviously, the procedure to compute the core tensor $\mathcal{S}_{k_{2}}$ is similar to the algorithm proposed in Algorithm \ref{TO-SVD}. Now, the rank-$\boldsymbol{k}$ representation can be written as
\begin{equation}\label{biaoshi}
\mathcal{A}_{\boldsymbol{k}}=\left( \mathcal{U}_{k_{2}} *_{3}\mathcal{S}_{k_{2}}*_{3}\mathcal{V}_{k_{2}}\right)  \times_{3} \boldsymbol{U}_{k_{1}}^{(3)}.
\end{equation}
If $I_{3}=1$, then Algorithm \ref{alg:ROSVD} reduces to Algorithm \ref{alg:r-svd}.
       \begin{algorithm}[htb]
	\caption{RO-SVD with power iterations}
	\label{alg:ROSVD}
	\begin{algorithmic}[1] 
	\REQUIRE $\mathcal{A}\in \mathbb{R}^{I_{1} \times I_{2} \times I_{3} }$, target truncation vector $\boldsymbol{k}=\left[ k_{1}; \boldsymbol{k}_{2}\right]$, $\boldsymbol{k}_{2}=\left[k_{21}, \ldots k_{2k_{1}}\right] ^{T}$,  oversampling parameter $p$, the first iteration parameter $q_{0}$, and the iteration vector $\boldsymbol{q}=\left(q_{1},q_{2},\ldots, q_{k_{1}} \right)$
	\ENSURE $\boldsymbol{U}_{k_{1}}^{(3)} \in \mathbb{R}^{I_{3} \times k_{1}}$, $\mathcal{U}_{k_{2}}\in \mathbb{R}^{I_{1} \times k_{2} \times k_{1} }$, $\mathcal{S}_{k_{2}}\in \mathbb{R}^{k_{2} \times k_{2} \times k_{1} }$, $\mathcal{V}_{k_{2}}\in \mathbb{R}^{k_{2} \times I_{2} \times k_{1} }$
	\STATE Generate $k_{1}+1$ Gaussian random  matrices  $\boldsymbol{\Omega} _{1} \in \mathbb{R}^{I_{1}I_{2} \times (k_{1}+p)}$ and $\boldsymbol{\Omega}_{2i} \in \mathbb{R}^{I_{2} \times (k_{2i}+p)};$
	\STATE Initialization: $\mathcal{U}_{k_{2}}$, $\mathcal{S}_{k_{2}}$, $\mathcal{V}_{k_{2}}$ are zero tensors of appropriate size;
	\STATE$\left[  \boldsymbol{U}_{k_{1}}^{(3)},\boldsymbol{S}_{k_{1}}^{(3)},\boldsymbol{V}\mathclap{\mathclap{_{k_{1}}}}^{(3)}\right]=\mathrm{rsvd}\left[ \boldsymbol{A}_{(3)},\boldsymbol{\Omega}_{1}, k_{1}, p, q_{0}\right] ;$
	\label{U1}
	\STATE for $i=1,2,\ldots ,k_{1}$, do
	\STATE \quad$\widehat{\boldsymbol{V}}\mathclap{_{i}}=\mathtt{reshape}(\boldsymbol{V}\mathclap{_{k_{1}}}^{(3)}(:,i),\left[I_{1},I_{2} \right] )$;	
	\label{hatV}	
	\STATE\quad $[\boldsymbol{U},\boldsymbol{S}, \boldsymbol{V}]=\mathrm{rsvd}(\widehat{\boldsymbol{V}}\mathclap{_{i}}, \boldsymbol{\Omega}_{2i}, k_{2i}, p, q_{i})$;
	\label{k2jieduan}
	\STATE\quad $\mathcal{U}_{k_{2}}(:,1:k_{2i},i)=\boldsymbol{U}$, $\mathcal{S}_{k_{2}}(1:k_{2i},1:k_{2i},i)=\boldsymbol{S}_{k_{1}}^{(3)}(i,i)\boldsymbol{S}$, $\mathcal{V}_{k_{2}}(1:k_{2i},:,i)=\boldsymbol{V}^{T}$;
	\STATE end for	
	\end{algorithmic}
	\end{algorithm}

\hskip 2em  We now present the error analysis for Algorithm \ref{alg:ROSVD}. There are two major difficulties here in extending the proofs of Theorem \ref{best}. For the probabilistic error analysis, it is important to note that at each step of the cycle, the partially truncated $\widehat{\boldsymbol{V}}_{i}$ is a random matrix. The elements on the diagonal of $\boldsymbol{S}_{k_{1}}^{(3)}$ are also derived from the R-SVD. Second, since the orthonormal matrix $\boldsymbol{U}_{k_{1}}^{(3)}$ here is generated by a randomized method, it no longer has the property $ \boldsymbol{U}^{(3)T}\boldsymbol{U}_{k_{1}}^{(3)}=\left[ \begin{array}{c}  \boldsymbol{I}_{k_{1}}\\\boldsymbol{0} \end{array}\right]$ as $\boldsymbol{U}^{(3)}$ in Theorem \ref{OSVD}. As a consequence,
using the orthogonal invariance to deal with the Frobenius norm directly does not work in deriving the expected error. In fact, we can provide an expected error bound by splitting the error into two parts.

\begin{mythm}\label{errorosvd}
Let $\mathcal{A}_{\boldsymbol{k}}$ be the output of Algorithm \ref{alg:ROSVD} with target truncation parameter $\boldsymbol{k}=\left[ k_{1};\boldsymbol{k}_{2}\right] $ satisfying $k_{1} \leq r_{1}$, $k_{2} \leq r_{2}$, oversampling parameter $p\geq 2$,  iteration count $q_{0}$, iteration vector $\boldsymbol{q}$ and  Gaussian random  matrices set $\boldsymbol{\Omega}_{2}=\left\lbrace\boldsymbol{\Omega}_{21},\ldots ,\boldsymbol{\Omega}_{2k_{1}} \right \rbrace $. Suppose $\mathcal{U}_{k_{2}}$, $\boldsymbol{U}_{k_{1}}^{(3)}$ are obtained from Algorithm \ref{alg:ROSVD}. Then, the approximation error in expectation satisfies
\begin{equation}
\begin{aligned}
\mathbb{E}_{\Omega_{1},\boldsymbol{\Omega}_{2}}\left\|  \mathcal{A}-\mathcal{A}_{\boldsymbol{k}}\right\| _{F} &\leq  \left[ \left(1+\frac{k_{1}}{p-1} \tau_{k_{1}}^{4 q_{0}}\right)\left(\sum_{i>k_{1}}^{r_{1}} \sum_{j\geq1}^{r_{2}}s_{jji}^{2}\right)\right]^{\frac{1}{2}}\\
&+\left[ \sum_{i=1}^{k_{1}}\left(1+\frac{k_{2i}}{p-1}\left(\tau_{k_{2i}}^{(i)}\right)^{4 q_{i}}\right)\left(\sum_{j\geq k_{2i}}^{r_{2}}s_{jji}^{2}\right)\right] ^{\frac{1}{2}},
\end{aligned}
\end{equation}
where $\tau_{k_{1}}$ is the singular value gap of  $\boldsymbol{A}_{(3)}$ and $\tau_{k_{2i}}^{(i)}$ is the singular value gap of  $\widetilde{\mathcal{A}}(:,:,i)$.
\end{mythm}
\begin{proof}
It is straightforward to show that
\begin{equation*}
\begin{aligned}
&\mathbb{E}_{\Omega_{1},\boldsymbol{\Omega}_{2}}\left\|  \mathcal{A}-\mathcal{A}_{\boldsymbol{k}}\right\| _{F} = \mathbb{E}_{\Omega_{1},\boldsymbol{\Omega}_{2}}\left\|  \mathcal{A}- \left[  \mathcal{U}_{k_{2}}*_{3}  \mathcal{U}_{k_{2}}^{T}*_{3} \left( \mathcal{A}\times_{3}\boldsymbol{U}_{k_{1}}^{(3)T}\right)\right] \times_{3}\boldsymbol{U}_{k_{1}}^{(3)}  \right\| _{F}	\\
&=\mathbb{E}_{\Omega_{1},\boldsymbol{\Omega}_{2}}\left\|  \mathcal{A}- \mathcal{A}\times_{3}\boldsymbol{U}_{k_{1}}^{(3)T}\times_{3}\boldsymbol{U}_{k_{1}}^{(3)}+ \mathcal{A}\times_{3}\boldsymbol{U}_{k_{1}}^{(3)T}\times_{3}\boldsymbol{U}_{k_{1}}^{(3)}- \left(  \mathcal{U}_{k_{2}}*_{3}  \mathcal{U}_{k_{2}}^{T}*_{3} \widehat{\mathcal{A}}\right)  \times_{3}\boldsymbol{U}_{k_{1}}^{(3)} \right\| _{F}\\	
&\leq \mathbb{E}_{\Omega_{1}}\left\|\mathcal{A}-\mathcal{A}\times_{3}\left(\boldsymbol{U}_{k_{1}}^{(3)}\boldsymbol{U}_{k_{1}}^{(3)T}\right)\right\| _{F}+\mathbb{E}_{\Omega_{1},\boldsymbol{\Omega}_{2}}\left\|\widehat{\mathcal{A}}\times_{3}\boldsymbol{U}_{k_{1}}^{(3)}-  \left(  \mathcal{U}_{k_{2}}*_{3}  \mathcal{U}_{k_{2}}^{T}*_{3} \widehat{\mathcal{A}}\right)  \times_{3}\boldsymbol{U}_{k_{1}}^{(3)}  \right\| _{F},	
\end{aligned}
\end{equation*}
where we have used the fact that the first part does not depend on the second random matrix $\boldsymbol{\Omega} _{2}$. We tackle two parts separately.\\
$\mathbf{Part\,\uppercase\expandafter{\romannumeral1}}:$
Using Theorem \ref{matrix error} and H{\"o}lder's inequality \cite[Theorem 23.10]{Jacod2012}, we can write the expected error of the first part directly
\begin{equation}\label{part1}
\begin{aligned}
\mathbb{E}_{\Omega_{1}}\left\|\mathcal{A}-\mathcal{A}\times_{3}\left(\boldsymbol{U}_{k_{1}}^{(3)}\boldsymbol{U}_{k_{1}}^{(3)T}\right)\right\| _{F}
&\leq \left( \mathbb{E}_{\Omega_{1}}\left\|  \mathcal{A}\times_{3}\left( \boldsymbol{I}_{I_{3}}-\boldsymbol{U}_{k_{1}}^{(3)}\boldsymbol{U}_{k_{1}}^{(3)T}\right) \right\|^{2}  _{F}\right)^{\frac{1}{2}} \\
&=\left( \mathbb{E}_{\Omega_{1}}\left\|  \left( \boldsymbol{I}_{I_{3}}-\boldsymbol{U}_{k_{1}}^{(3)}\boldsymbol{U}_{k_{1}}^{(3)T}\right)\boldsymbol{A}_{(3)}\right\| ^{2} _{F}\right) ^{\frac{1}{2}}\\
&\leq\left[ \left(1+\frac{k_{1}}{p-1} \tau_{k_{1}}^{4 q_{0}}\right)\left(\sum_{i>k_{1}} \sigma_{i}^{2}\right)\right] ^{\frac{1}{2}}.
\end{aligned}
\end{equation}
By Theorem \ref{dengyu},  we can replace  $\sigma_{i}^{2}$ by $ \sum_{j\geq1}^{r_{2}}s_{jji}^{2}$.\\		
$\mathbf{Part\,\uppercase\expandafter{\romannumeral2}}:$
As for  the second part, since $\boldsymbol{U}_{k_{1}}^{(3)}$ has orthonormal columns,
\begin{equation*}
\left\|\widehat{\mathcal{A}}\times_{3}\boldsymbol{U}_{k_{1}}^{(3)}-  \left(\mathcal{U}_{k_{2}}*_{3} \mathcal{U}_{k_{2}}^{T}*_{3} \widehat{\mathcal{A}}\right) \times_{3}\boldsymbol{U}_{k_{1}}^{(3)}  \right\| _{F}\leq
\left\|	\widehat{\mathcal{A}}- \mathcal{U}_{k_{2}}*_{3} \mathcal{U}_{k_{2}}^{T}*_{3} \widehat{\mathcal{A}} \right\|^{2} _{F}.
\end{equation*}
Then, using Theorem \ref{matrix error} (keeping $\boldsymbol{\Omega}_{1}$ fixed) and the linearity of expectation, we have
\begin{equation*}
\begin{aligned}
\mathbb{E}_{\boldsymbol{\Omega}_{2}}\left\|\widehat{\mathcal{A}}- \mathcal{U}_{k_{2}}*_{3} \mathcal{U}_{k_{2}}^{T}*_{3} \widehat{\mathcal{A}} \right\|^{2} _{F}
&=\sum_{i=1}^{k_{1}}\mathbb{E}_{\Omega_{2i}}\left\|\left( \boldsymbol{I}_{I_{1}}- \mathcal{U}_{k_{2}}(:,:,i)\mathcal{U}_{k_{2}}^{T}(:,:,i)\right)\widehat{\sigma}_{i}\widehat{\boldsymbol{V}}_{i} \right\|^{2}_{F} \\
&\leq \sum_{i=1}^{k_{1}}\left(1+\frac{k_{2}}{p-1}\left(\tau_{k_{2i}}^{(i)}\right)^{4 q_{i}}\right)\left(\sum_{j>k_{2i}}^{r_{2}}(\widehat{\sigma}_{i}\sigma_{ij})^{2}\right),
\end{aligned}
\end{equation*}
where $\widehat{\sigma}_{i} $ is the $i$th singular value of the $\widehat{\boldsymbol{A}}_{(3)}$.

\hskip 2em We recall the definition of L{\"o}wner partial ordering \cite[Section 7.7]{Johnson2012}. Let $\boldsymbol{A}$, $\boldsymbol{B} \in \mathbb{R}^ {n\times n}$ be Hermitian; $\boldsymbol{A} \preceq \boldsymbol{B}$ means $\boldsymbol{B} - \boldsymbol{A}$ is positive semi-definite. Furthermore, $\lambda_{i}(\boldsymbol{A}) \leq \lambda_{i} (\boldsymbol{B})$ for $i = 1,2,\ldots, n$, where $\lambda$ is the eigenvalue of the matrix. Notice $\boldsymbol{U}_{k_{1}}^{(3)}\boldsymbol{U}_{k_{1}}^{(3)T} $ is  a projector so that
\begin{equation*}
\widehat{\boldsymbol{A}}_{(3)}^{T}\widehat{\boldsymbol{A}}_{(3)}=\left( \boldsymbol{U}_{k_{1}}^{(3)T}\boldsymbol{A}_{(3)}\right) ^{T}\left( \boldsymbol{U}_{k_{1}}^{(3)T}\boldsymbol{A}_{(3)}\right) \preceq \boldsymbol{A}_{(3)}^{T}\boldsymbol{A}_{(3)},
\end{equation*}
and the singular values of $\widehat{\boldsymbol{A}}_{(3)}$ satisfy 
\begin{equation*}
\widehat{\sigma}_{i} \leq \sigma_{i} ,\quad\text{for}\qquad i=1,2,\ldots,k_{1}.
\end{equation*}
Applying H{\"o}lder's inequality gives
\begin{equation}\label{part2}
\mathbb{E}_{\boldsymbol{\Omega}_{2}}\left\|\widehat{\mathcal{A}}- \mathcal{U}_{k_{2}}*_{3} \mathcal{U}_{k_{2}}^{T}*_{3} \widehat{\mathcal{A}} \right\| _{F}\leq \left[ \sum_{i=1}^{k_{1}}\left(1+\frac{k_{2i}}{p-1}\left(\tau_{k_{2i}}^{(i)}\right)^{4 q_{i}}\right)\left(\sum_{j>k_{2i}}^{r_{2}}(\sigma_{i}\sigma_{ij})^{2} )\right)\right]^{\frac{1}{2}}.
\end{equation}
Combining (\ref{part1}), (\ref{part2}) and $s_{jji}=\sigma_{i}\sigma_{ij}$ gives the conclusion.
\end{proof}
\subsection{Computational  Complexity and Memory  Cost}
\hskip 2em  We now discuss the computational cost of Algorithm \ref{TO-SVD} and Algorithm \ref{alg:ROSVD} {\color{blue},} and compare them  against the O-SVD, the RT-SVD proposed by Zhang \emph{et al}.\cite[Algorithm 6]{zhang2018} and the R-HOSVD proposed by Minster \emph{et al}.\cite[Algorithm 3.1]{2020Randomized}. We assume that the tensors are dense and the target truncation terms $k_{1}$ and $k_{2}$ are sufficiently small, i.e., $k_{1} \ll r_{1}$, $k_{2} \ll r_{2}$, so that we can neglect the computational cost of the QR factorization and the truncation steps of the R-SVD algorithm. The dominant cost of Algorithm \ref{alg:ROSVD} lies in  computing  a total of $k_{1}+1$ R-SVD, while the TO-SVD algorithm requires to compute the full SVD, which results in an expensive computational cost.

\hskip 2em Recall that the RT-SVD algorithm consists of transforming the tensor to the Fourier domain and applying the R-SVD to each frontal slice of the tensor. The R-HOSVD algorithm has three main steps including  multiplying each mode unfolding with a Gaussian random  matrix, computing an approximation to the column space and then forming the core tensor. The storage and computational cost of the TO-SVD,  RO-SVD, RT-SVD and RHOSVD algorithms are summarized in Table \ref{cost}. Each algorithm takes a core tensor $\mathcal{S}$ of the same size except the RT-SVD. The table includes the costs for a general third-order tensor $\mathcal{A} \in \mathbb{R}^{I_{1} \times I_{2} \times I_{3}}$ with the first iteration parameter $q_{0}$ and the iteration vector $\boldsymbol{q}=\left(q_{1},q_{2},\ldots q_{k_{1}} \right)$ ($k_{1}=I_{3}$ for the RT-SVD, target rank $\boldsymbol{k}_{3}=(k_{2},k_{2},k_{1})$ for the R-HOSVD). For a more intuitive comparison of the computational cost we assume that $k_{21}=\cdots=k_{2k_{1}}=k_{2}$. Since by assumption $k_{1} \ll r_{1}$, $k_{2} \ll r_{2}$, the RO-SVD is expected to be much faster than all four other algorithms.
\begin{table} [h]
	\centering
	\caption{Comparison of O-SVD, TO-SVD, RO-SVD, RT-SVD, R-HOSVD}
	\label{cost}
	\begin{tabular}{c|c|c}
		\hline
{Algorithm} & Computational Cost & Storage Cost \\
                 \hline
  O-SVD &$\mathcal{O}\left(I_{1}I_{2}I^{2}_{3}+R_{3}I_{1}I_{2}I_{3}+R_{3}I_{1}I^{2}_{2}\right)$&$R_{3}(I_{2}I_{1}+I_{2}I_{2}+I_{3}+I_{2})$ \\
                    \hline
TO-SVD &$\mathcal{O}\left(I_{1}I_{2}I^{2}_{3}+k_{1}I_{1}I^{2}_{2}\right)$&$k_{1}k_{2}I_{1}+k_{1}k_{2}I_{2}+k_{1}I_{3}+k_{1}k_{2}$ \\
	\hline
RO-SVD &$\mathcal{O}\left( (2q_{0}+1)k_{1}I_{1}I_{2}I_{3}+\sum_{i=1}^{k_{1}}(2q_{i}+1)k_{2}I_{1}I_{2}\right)$ &$k_{1}k_{2}I_{1}+k_{1}k_{2}I_{2}+k_{1}I_{3}+k_{1}k_{2}$ \\
		\hline
RT-SVD  &$\mathcal{O}\left(I_{1}I_{2}I_{3} \log I_{3}+\sum_{i=1}^{I_{3}}(2q_{i}+1)k_{2}I_{1}I_{2} \right)$ &$k_{2}(I_{1}I_{3}+I_{2}I_{3}+I_{3})$ \\
		\hline
R-HOSVD &$\mathcal{O}(\sum_{i=1}^{3}(2q_{i}+1)k_{3i}I_{1}I_{2}I_{3}+k_{1}I_{1}I_{2}+k_{1}k_{2}I_{2}+k_{1}k_{2}^{2})$ &$k_{1}k_{2}^{2}+k_{2}(I_{1}+I_{2})+I_{3}k_{1}$ \\
		\hline
		\end{tabular}
	\end{table}

\section{Numerical  Examples}\label{Numericalsection}
\hskip 2em In order to evaluate Algorithm \ref{TO-SVD} and Algorithm \ref{alg:ROSVD}, we present the numerical results by comparing them with truncation methods such as  TT-SVD, THOSVD and the above mentioned randomized algorithm of tensor singular value decomposition: RT-SVD, RHOSVD.  All the computations are based on the Matlab Tensor Toolbox \cite{TensorToolbox2017} and Tensor-Tensor Product Toolbox\cite{Toolbox2018}. Our results were run in Matlab R2020b on a Lenovo computer with AMD Ryzen 5 3500U processor and 12 GB RAM. We use the HOSVD algorithm to show that the tensor is orientation-dependent.
 To the best of our knowledge, there is no specific approach for the selection of all truncation parameters in fixed-rank random tensor algorithms. The  oriented tensors have high correlation among frontal slices and we can estimate the rank of the orthonormal matrix in the RO-SVD in advance. We could set $k_{1}$ much smaller than $I_{3}$. The optimal value of $k_{1}$ is the number of a basis in the space spanned by all frontal slices of the third-order oriented  tensor.
 Since both the TT-SVD and the RT-SVD use the same truncation parameters $k_{2}$ for each frontal slice, we also use the same truncation parameters $k_{21}=\cdots=k_{2k_{1}}=k_{2}$ for TO-SVD and RO-SVD in the experiments. In addition, the computational time of each method was measured in seconds. The results of each experiment were averaged three times.

\hskip 2em We employ four indices, i.e., the relative error, the compression ratio, the peak-signal-to-noise ratio (PSNR) and the structural similarity index (SSIM), to evaluate the performance of image compression algorithm. If $\mathcal{A} \in \mathbb{R}^{I_{1} \times I_{2} \times I_{3}}$ is the original tensor and $\widehat{ \mathcal{A}}$ is a low-rank approximation of $\mathcal{A}$, then the relative approximation error  and the compression ratio are respectively given by
\begin{equation*}
\mathrm{Err}=\frac{\left\| \mathcal{A}-\widehat{ \mathcal{A}}\right\|_{F} }{\left\| \mathcal{A} \right\|_{F} }\,\text{and}\,\mathrm{Ratio}=\frac{I_{1} \times I_{2} \times I_{3}}{\mathrm{storage}\,\mathrm{cost}}.
\end{equation*}
The PSNR is given by:
$$
\mathrm{P S N R}=20 \log _{10} \frac{\max (\mathcal{A})}{\sqrt{\mathrm{M S E}}},
$$
where the  MSE of tensor $\mathcal{A}$ is as follows:
$$
\mathrm{M S E}=\frac{\|\mathcal{A}-\hat{\mathcal{A}}\|_{F}^{2}}{I_{1} \times I_{2} \times I_{3}}.
$$
According to research, a PSNR value above 40 for the pixel component of an image is an indication of very good quality (i.e., the restored frame is very close to the original frame). If the PSNR is between 30 and 40, then the image quality is usually good (i.e., the distortion in the restored image is noticeable but still acceptable). If it is between 20 and 30 then the image quality is poor, and finally, images with a PSNR below 20 are not acceptable. SSIM \cite{image2004} measures the similarity between the original image and the reconstructed image on structural consistency. The equation of SSIM is  below:
$$
\mathrm{SSIM}(a, \hat{a})=\frac{\left(2 \mu_{a} \mu_{\hat{a}}+C_{1}\right)\left(2 \sigma_{a \hat{a}}+C_{2}\right)}
{\left(\mu_{a}^{2}+\mu_{\hat{a}}^{2}+C_{1}\right)\left(\sigma_{a}^{2}+\sigma_{\hat{a}}^{2}+C_{2}\right)},
$$
where $\mu_{a}$ and $\mu_{\hat{a}}$ are mean intensities, $\sigma_{a}$ and $\sigma_{\hat{a}}$ are standard deviations, $C_{1}$ and $C_{2}$ are default values. Covariance $\sigma_{a\hat{a}}$ is calculated as follows:
$$
\sigma_{a \hat{a}}=\frac{1}{I_{1} I_{2} I_{3}-1} \sum_{ijk}^{}\left(a_{ijk}-\mu_{a}\right)\left(\hat{a}_{ijk}-\mu_{\hat{a}}\right).
$$
Obviously, the SSIM is a number between 0 and 1. The larger the SSIM, the smaller the difference between the two images.

\subsection{Hyperspectral  Image}
\hskip 2em In this subsection, we test a hyperspectral image— Salinas\cite{Indianpines}. This scene was collected by the 224-band AVIRIS sensor over Salinas Valley, California, and is characterized by high spatial resolution (3.7-meter pixels). The area covered comprises $512$ lines by $217$ samples 224 available spectral reflectance bands in the wavelength range. Hence, the size of the resulting tensor is $512\times217\times224$. Denote by $\mathcal{A}$ the tensor of the testing data. Under the relative error tolerance 0.005, (Matlab command: $\mathrm{hosvd}\left( \mathcal{A},0.005\right)$), the size of the nonzero part of the core tensor is  $423\times203\times 32$. Hence, this is a well-oriented tensor. And the running time of a full O-SVD is 6.2432s.

\hskip 2em  We take the first iteration parameter $q_{0}=1$, the target truncation term $k_{2}$ fixed at 80, the oversampling parameter $p=5$ and the second iteration parameters $q_{i}=1$. The relative errors for varying $k_{1}$ between 5 and 40 are plotted in Figure \ref{k1duibi}, where we can see the errors of the RO-SVD are quite close to the TO-SVD. Since $\mathcal{A}$ is a well-oriented tensor, the errors level off after $k_{1}=30$. The errors of TO-SVD and TT-SVD are very close, indicating that it is possible to compress the third dimension of  $\mathcal{A}$ to $k_{1}$ without being affected. The error of THOSVD is larger for the same  size of core tensor.

\hskip 2em  Then, we choose  $k_{1}$ to be 35 and allow $k_{2}$ to vary between 30 and 100 for simplicity. We track the relative errors of the RO-SVD for the case that the iteration parameters $q_{i}$ ($i=1,2, \ldots, k_{1}$) are equal, i.e., $q_{1}= \cdots =q_{k_{1}}=q$. Figure \ref{k2duibi} shows that the errors for varying $q$ have similar convergence trajectories and $q=1$ is enough for practical use.  Moreover, as $q$ increases, a much more accurate approximation is yielded.
In addition, we give error curves for TT-SVD, THOSVD and their corresponding randomized algorithms with $q = 1$. It is shown that under the same conditions, the tensor singular value decomposition  based on the Tucker product has a much larger error than the other two tensor decompositions, both for the exact and the randomized algorithms.

\hskip 2em Figure \ref{ratiobutong} shows the time, relative error, PSNR and SSIM of the three randomized algorithms at different compression ratios.
We choose  $k_{1}$ to be 35, the oversampling parameter $p=5$, all the iteration parameters $q=1$ and $k_{2}$ varies with the compression ratio.
Figure \ref{ratiobutong} shows that the RO-SVD outperforms the RT-SVD in all four indices.
In Table \ref{ratio}, we record the time, relative error, PSNR and SSIM of four algorithms (TO-SVD, RO-SVD, RT-SVD,  and RHOSVD) at $Ratio=40$.
From  Figure \ref{ratiobutong} and Table \ref{ratio},  we find that the proposed randomized algorithm has similar results to the exact algorithm,
illustrating the effectiveness of the RO-SVD. As for the running time, the RO-SVD is usually much faster than TO-SVD and O-SVD respectively.
Compared with other randomized algorithms, the PSNR and SSIM of RHOSVD and RO-SVD are within the acceptable range, but the former about 6 times longer than the latter when $Ratio=40$. As both the time and relative error of  the RO-SVD are reduced by half compared to the RT-SVD, the RO-SVD is superior to  the RT-SVD at the same compression ratio.
This  gives us a suggestion for choosing $k_{2}$: in addition to setting the value of $k_{2}$ directly, we can also set the parameter $k_{2}$ according to the desired compression ratio.
\begin{figure}[htbp]
	\begin{minipage}[t]{0.45\linewidth}
		\centering
			\includegraphics[scale=0.45]{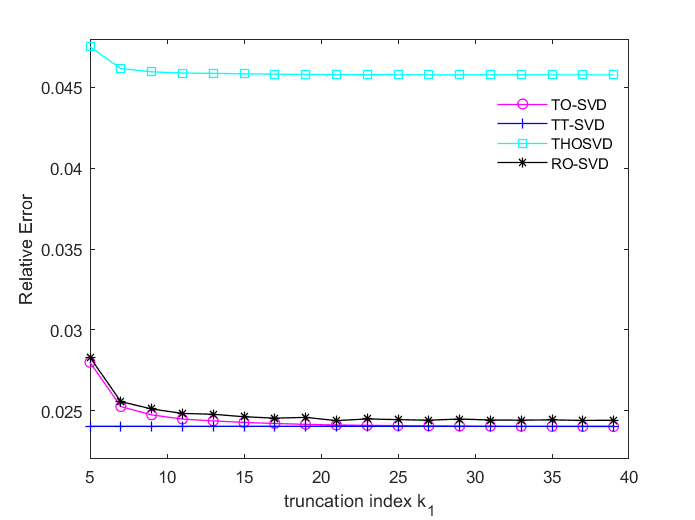}
		\caption{Relative errors of four algorithms with $k_{2}=80$}
		\label{k1duibi}
	\end{minipage}%
	\quad
	\begin{minipage}[t]{0.45\linewidth}
		\centering
			\includegraphics[scale=0.45]{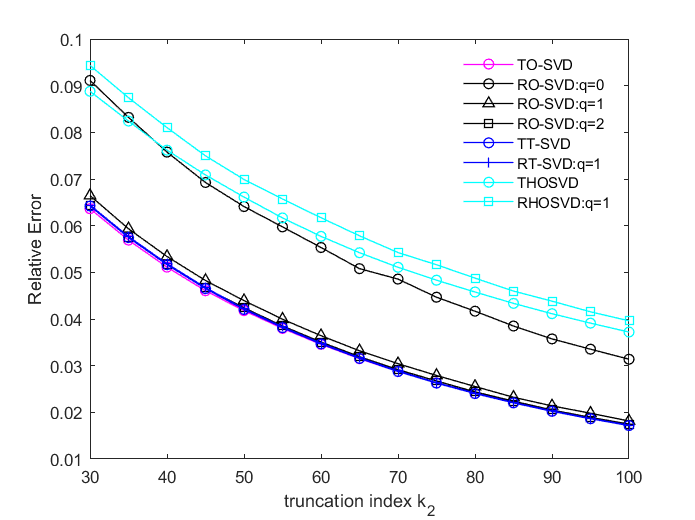}
		\caption{Relative errors of four algorithms with $k_{1}=35$}
		\label{k2duibi}
	\end{minipage}
\end{figure}

\begin{figure}[htbp]
	\centering
	\subfigure[Time]{
		\includegraphics[width=7.5cm]{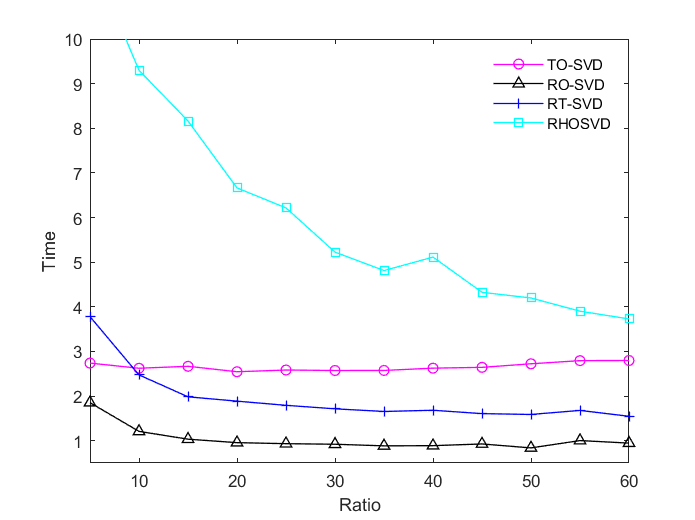}
	}
	\quad
	\subfigure[Relative errors]{
		\includegraphics[width=7.5cm]{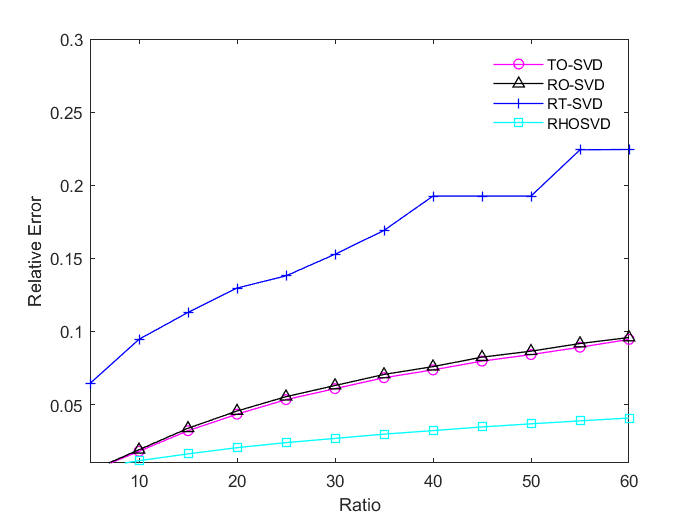}
	}
	\quad
	\subfigure[PSNR]{
		\includegraphics[width=7.5cm]{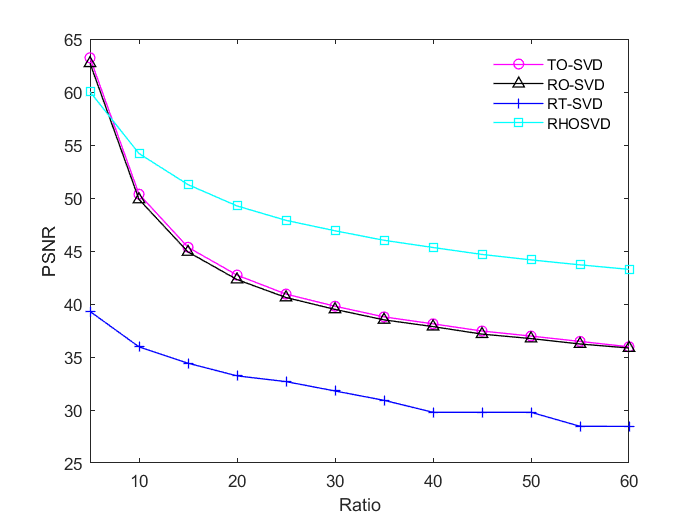}
	}
	\quad
	\subfigure[SSIM]{
		\includegraphics[width=7.5cm]{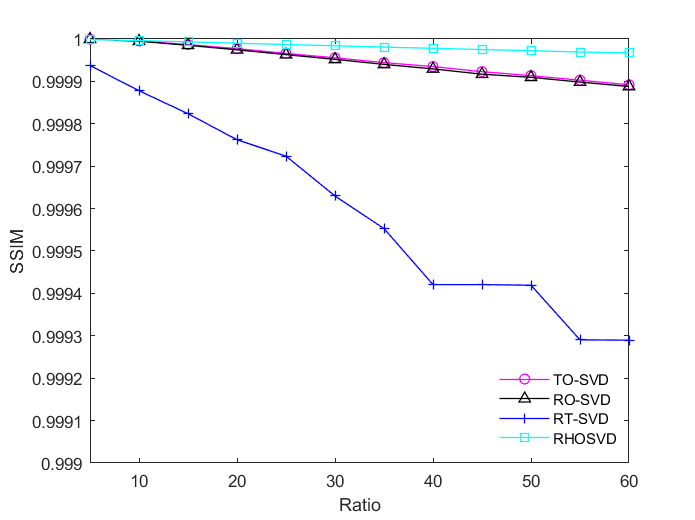}
	}
	\caption{Comparison results of  four algorithms (TO-SVD, RO-SVD, RT-SVD,  and RHOSVD) on Salinas.}
	\label{ratiobutong}
\end{figure}

\begin{table} [htbp]
\centering
\caption{ Comparison of four algorithms with  $Ratio=40$ }
\label{ratio}
\begin{tabular}{ccccc}
			\hline
Algorithm  & time& Err&PSNR&SSIM\\
		\hline
RT-SVD &1.6819 &0.1926  & 29.7846&0.9994 \\
			\hline
RHOSVD  &5.1130 & 0.0322 &45.3095 & 0.9999           \\
			\hline
TOSVD   &2.6429 & 0.0737 &38.1286 & 0.9999             \\
	                     \hline
RO-SVD   & 0.8886& 0.0758 &37.8799&   0.9999       \\		
                          	\hline
\end{tabular}
\end{table}

\subsection{Video}
\begin{figure}[htbp]
	\centering
	\subfigure[Time]{
		\includegraphics[width=7.5cm]{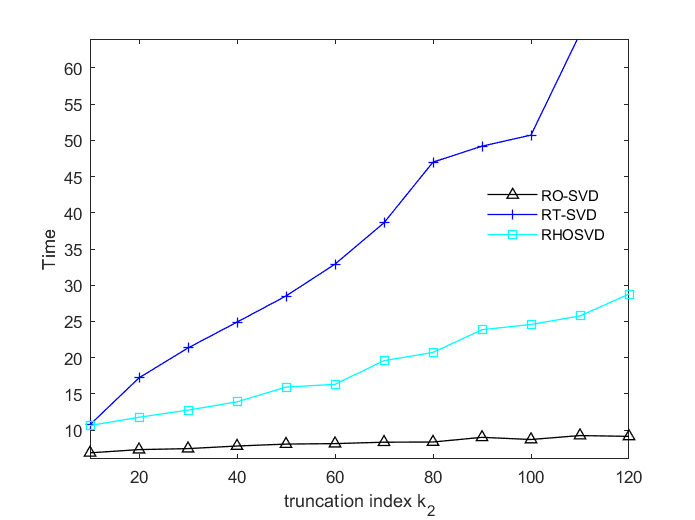}
	}
	\quad
	\subfigure[Relative errors]{
		\includegraphics[width=7.5cm]{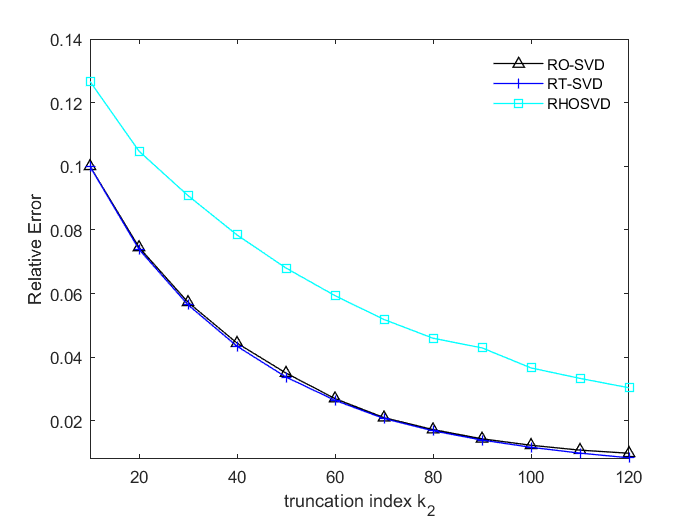}
	}
	\quad
	\subfigure[PSNR]{
		\includegraphics[width=7.5cm]{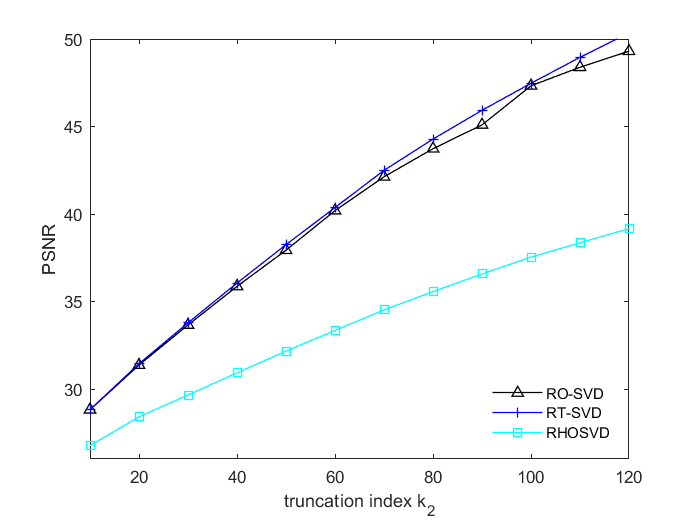}
	}
	\quad
	\subfigure[SSIM]{
		\includegraphics[width=7.5cm]{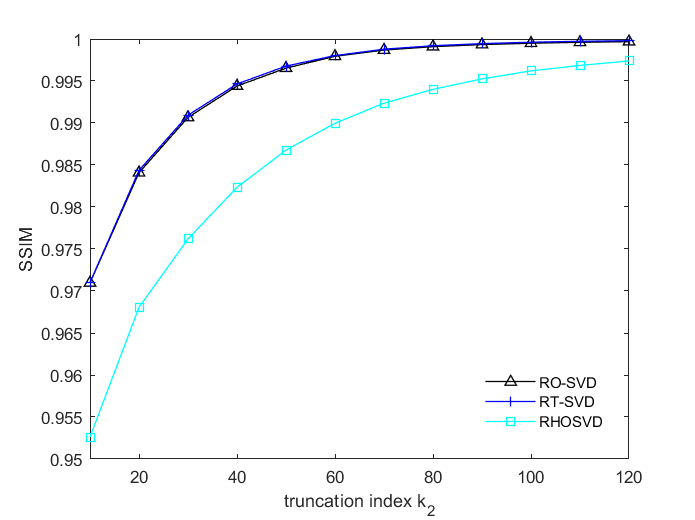}
	}
	\caption{Comparison results of three tensor randomized methods (RO-SVD, RT-SVD, and RHOSVD) on the video.}
	\label{k2butong}
\end{figure}
\begin{figure}[htbp]
	\centering
	\subfigure[Original video]{
		\includegraphics[width=7.5cm]{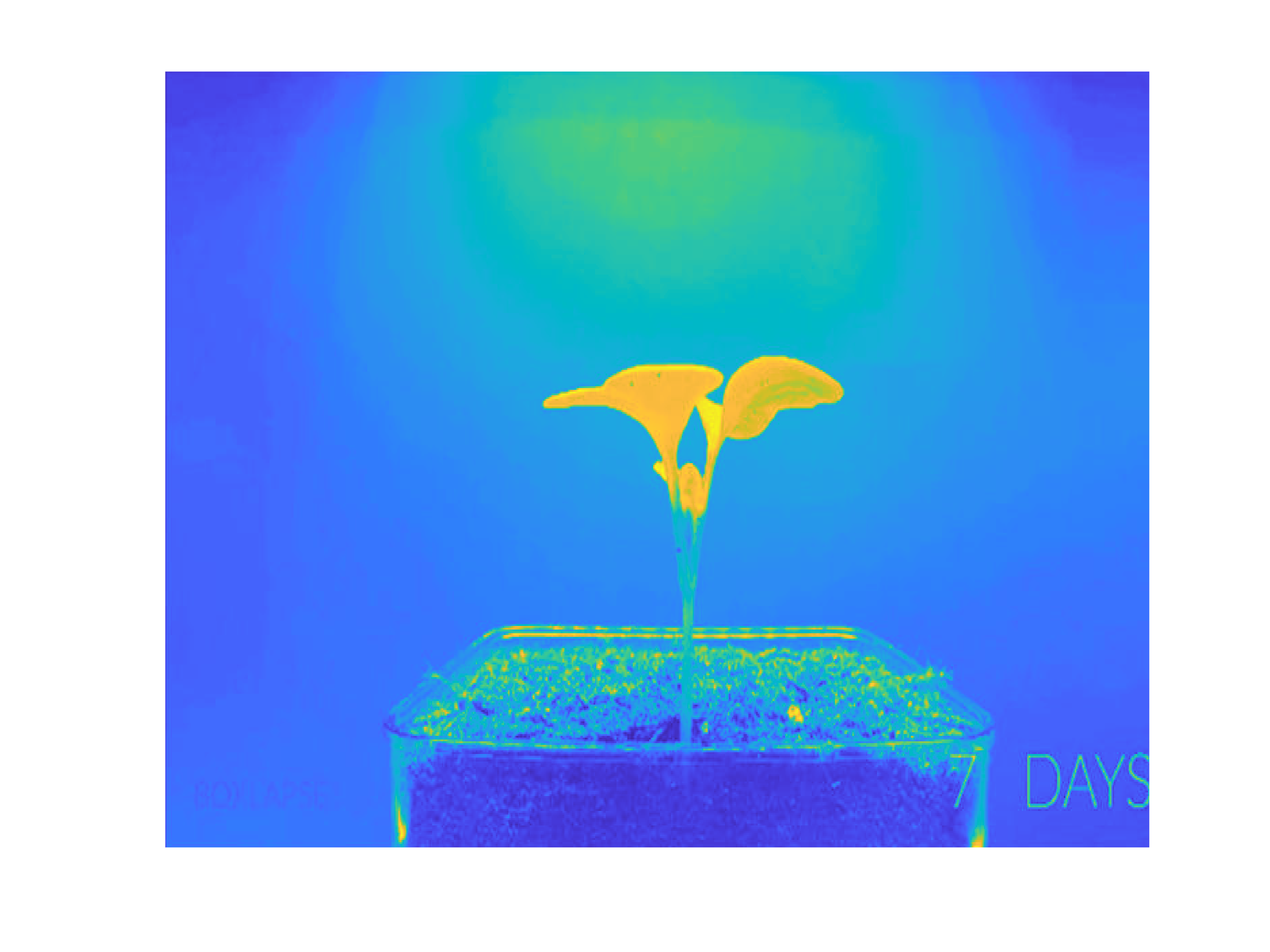}
	}
	\quad
	\subfigure[RT-SVD ]{
		\includegraphics[width=7.5cm]{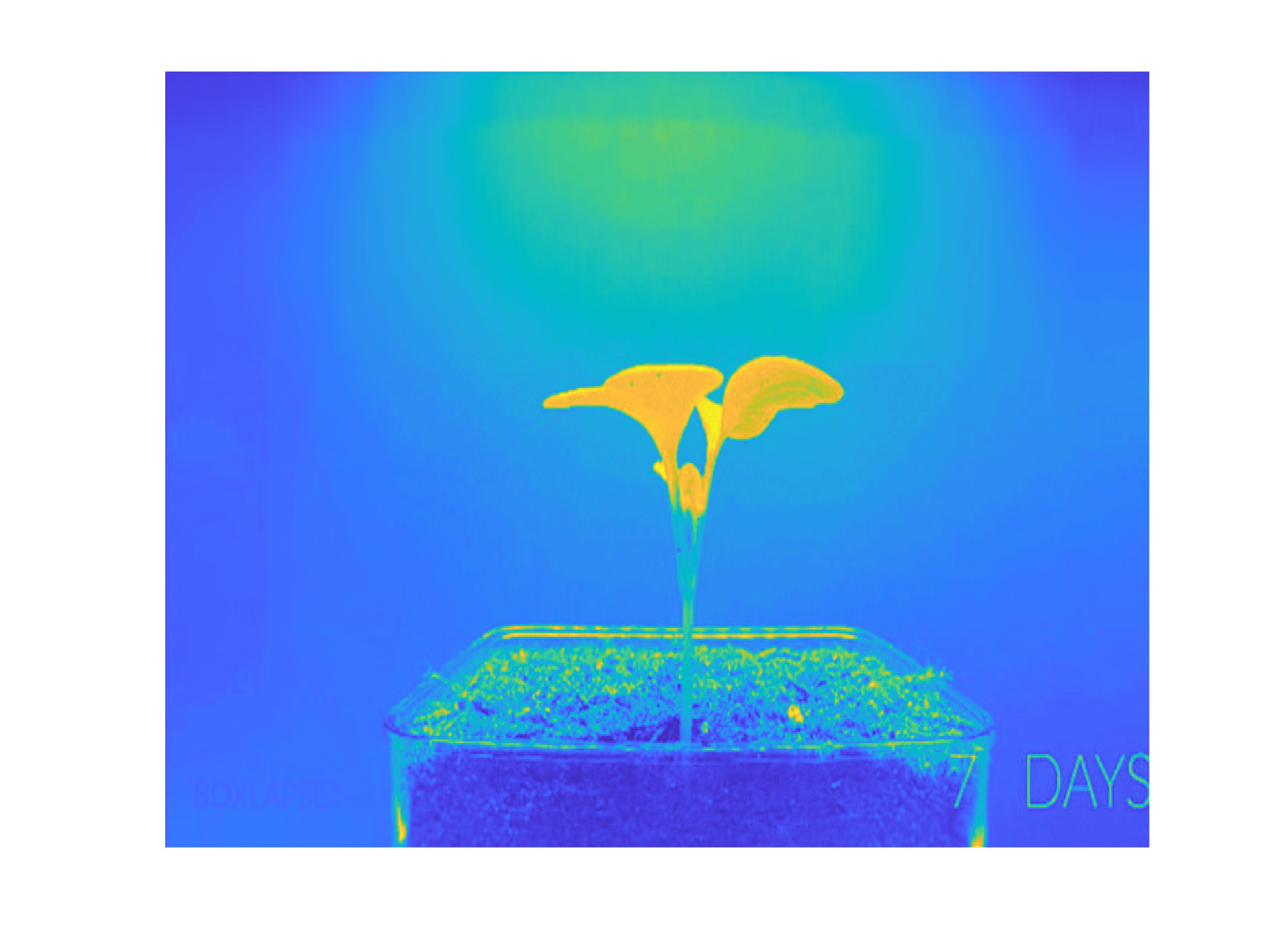}
	}
	\quad
	\subfigure[RHOSVD ]{
		\includegraphics[width=7.5cm]{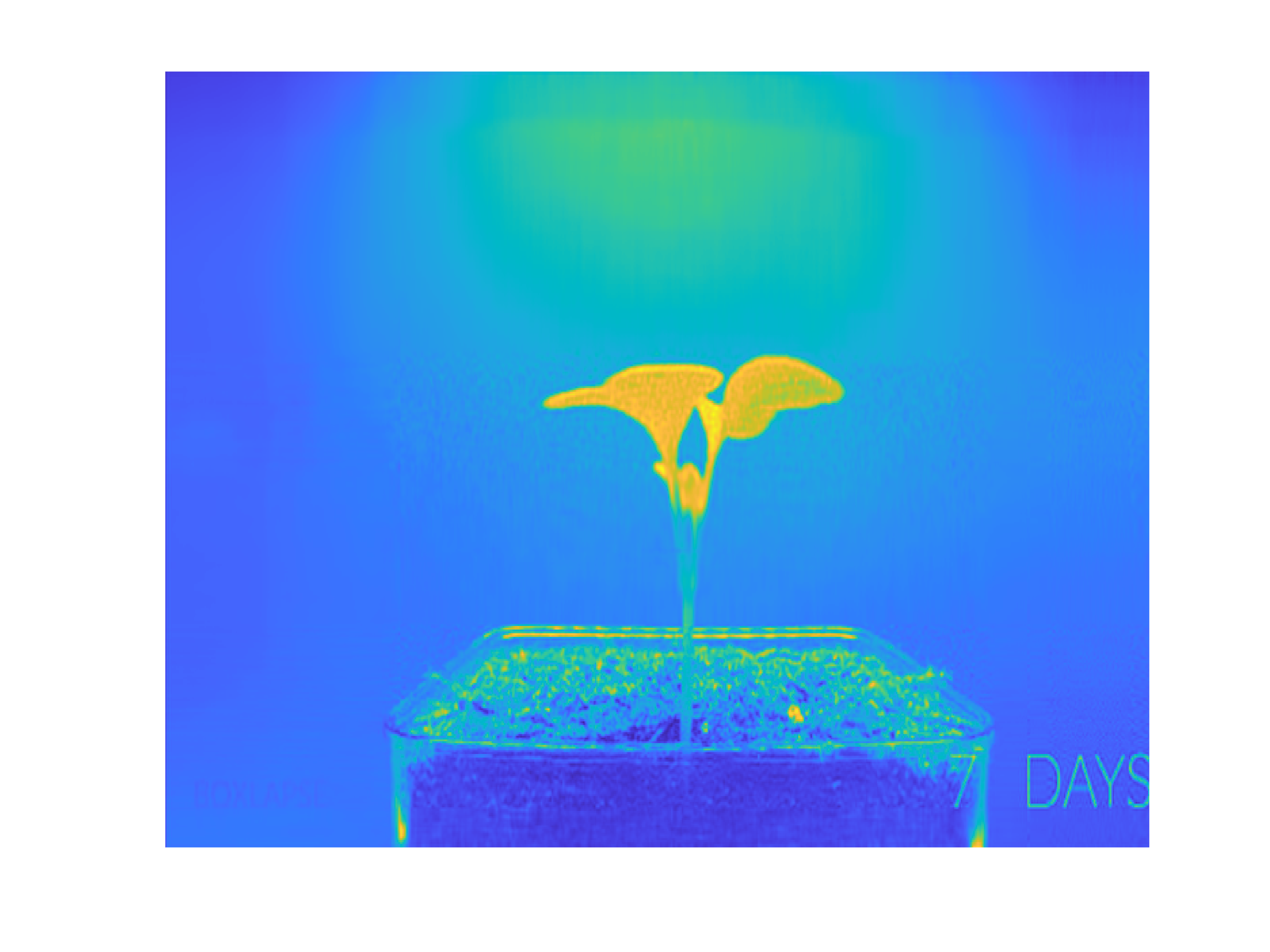}
	}
	\quad
	\subfigure[RO-SVD ]{
		\includegraphics[width=7.5cm]{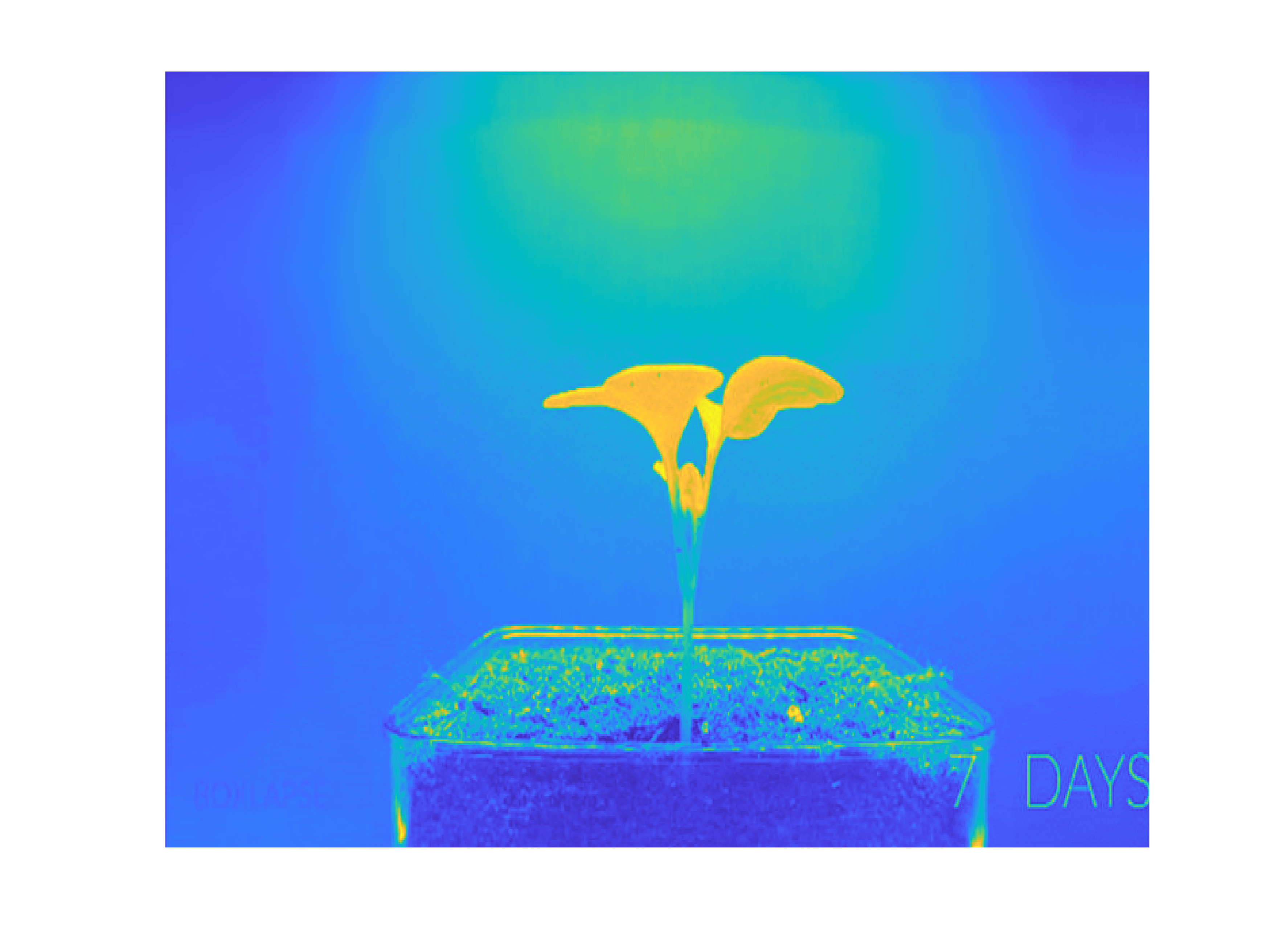}
	}
	\caption{Comparison results of three tensor randomized methods (RO-SVD, RT-SVD, and RHOSVD) with the original image}
	\label{tupian2}
\end{figure}
\begin{table}[htbp]
	\centering
	\caption{ Comparison of RT-SVD , RHOSVD and RO-SVD  with $k_{1}=70$ , $k_{2}=90$ }
	\label{duibi}
	\begin{tabular}{ccccc}
		\hline
		Algorithm  & time& Err&PSNR&SSIM \\
		\hline
		RT-SVD&49.1988 &0.0139  & 45.5254&0.9993 \\
		\hline
		RHOSVD  &23.8538  &0.0424&36.5753 & 0.9953          \\
		\hline
		RO-SVD & 8.9908 & 0.0144&45.0915 &   0.9993       \\		
		\hline
	\end{tabular}
\end{table}
\hskip 2em The second example is a video. The dataset $\mathcal{B}$ is a  video from Tencent$\footnote{https://v.qq.com/x/page/v3237ztwzs7.html}$, with size $ 424\times 726 \times 500$ (height$\times$width$\times$frames). See Figure \ref{tupian2}(a). Most regions of the frames are stable, while a radish is growing delicately. Under the relative error tolerance 0.01, the size of the nonzero part of the core tensor is  $198 \times 351 \times 71$, which is a well-oriented tensor. And the running time of a full O-SVD is 58.6212s.

 \hskip 2em We run three tensor randomized methods (RO-SVD, RT-SVD, and RHOSVD) with increasing the same target truncation term $k_{2}$ and oversampling parameter $p=5$ while setting the target rank $k_{1}=70$. Here we only show the  results with the power parameter $q = 1$, as for many applications this already achieves sufficient accuracy. We compare the time, relative error, PSNR and SSIM  of the RT-SVD, the RHOSVD and the RO-SVD algorithms with different target truncation term $k_{2}$, which are shown in Figure \ref{k2butong}. In Table \ref{duibi}, we record the results of  three algorithms at $k_{2}=90$ from Figure \ref{k2butong}. From the results in Figure \ref{k2butong}(a) and (b),
 we see that the RO-SVD algorithms is far superior to them in terms of time cost while keeping the accuracy and  its advantage  becomes more apparent with the increase of  $k_{2}$. Both the RT-SVD and the RO-SVD algorithm perform better than RHOSVD, as shown in Figure \ref{k2butong}(c) and (d). In Table \ref{duibi},
we can see the RO-SVD is about 6 times  faster than the implementation of RT-SVD, and the compression ratio of the former is 7.1 times (21.12 vs. 2.97) better than that of the latter. Figures \ref{tupian2}(b), \ref{tupian2}(c) and \ref{tupian2}(d) show the results of the RT-SVD, the RHOSVD and the RO-SVD respectively when the target truncation term $k_{2} = 90$. We can see the performance of the RO-SVD is the best with the same target truncation term $k_{2}$ while it has significantly lower computational cost.

\subsection{ Synthetic  Oriented Tensor}

\hskip 2em  For further reflection on the characteristics of the RO-SVD for large-scale oriented tensors, two kinds of synthetic tensors are tested standing for different distribution patterns of singular values:\\
 $\bullet$ Tensor 1 (slow decay): $\mathcal{A}=\left(\mathcal{U} *_{3} \mathcal{S} *_{3} \mathcal{V}\right) \times_{3} \boldsymbol{U}^{(3)}$, where $\boldsymbol{U}^{(3)}$, $\mathcal{U}(:,:,i)$ and $\mathcal{V}(:,:,i)^T$ are randomly drawn
  matrices with orthonormal columns, and the diagonal matrix $\mathcal{S}(:,:,i)$ has diagonal elements $\sigma_{jji} =1/(i+j)^{2}$.
  Furthermore, if $\boldsymbol{U}^{(3)}$ is a matrix with far fewer columns than rows, the resulting tensor $\mathcal{A}$ has to be a well-oriented tensor.
  \\
 $\bullet$ Tensor 2 (fast decay): $\mathcal{A}$ is formed just like Tensor 1, but the diagonal
    elements of $\mathcal{S}(:,:,i)$ are given by $\sigma_{jji} =e^{-j-i/7}$. It reflects a fast decay of singular values.

 \hskip 2em For each kind, we first generate a $1000 \times 1000 \times 300$ oriented tensor with $\mathtt{rank}_{3}(\mathcal{A})= 30$,
 for which we compare the relative errors and time of the proposed techniques for different $\boldsymbol{k}_{2}$ and $q$.
 Let $k_{1} = 30$, $p = 5$ and $k_{2i}$ be randomly generated positive integers in $[0,20], [20,40], \ldots, [180, 200]$, respectively.
 The  results are plotted in Figure \ref{tensor12}.
 We observe that the RO-SVD runs at approximately three times the speed of the TO-SVD within an acceptable error margin.
 By the power scheme,  the errors of the RO-SVD can be remarkably reduced.
 We notice that $q = 1$ suffices in practice since it produces indistinguishable results with $q=2$.

 \hskip 2em Then, we compare the performance of the RT-SVD, RHOSVD and RO-SVD algorithms
 with equal $k_{2i},$ oversampling parameter $p = 5$ and power parameter $q = 1$
 while setting the target rank $k_{1} = 30$, which are shown in Figure \ref{tensor12OTH}.
 It is obvious that the RO-SVD shows advantages in both  accuracy and time cost.
 Specifically, our algorithm takes roughly one-sixth of the time required for the other two fixed-rank randomized algorithms.
 Moreover, the RO-SVD achieves better compression ratio compared to the RT-SVD with $k_{2i} = 80$ (62.35 vs 6.24).

  \hskip 2em Finally, we conduct an experiment to compare the efficiency of the proposed algorithms with varying sizes. We generate oriented tensors with $\mathtt{rank}_{3}(\mathcal{A})=I_{3}/10 $ and summarize the  running time and relative errors of the five algorithms for different dimensions in Table \ref{hechengduibi}. The RO-SVD shows overwhelming advantages in computational efficiency.
  Specially, when the tensor size comes to $1000 \times 1000 \times 400$, all the other existing algorithms  lose competitiveness because of timeouts and  memory limits.
 \begin{table}[htbp]
 	\centering
 	\caption{ Comparison of five algorithms for  different $\mathcal{A}$ \\
 		with $k_{1}=I_{3}/10$, $k_{2}=200$, $q=1$ and $p=5$}
 	\label{hechengduibi}
 	\begin{tabular}{cccccccccc}
 		\hline
 		 Algorithm &O-SVD& \multicolumn{2}{ c }{TO-SVD}&\multicolumn{2}{ c }{RO-SVD}&\multicolumn{2}{ c }{RT-SVD}&\multicolumn{2}{ c }{RHOSVD} \\
 		\hline
 		
                 $I_{1} \times I_{2} \times I_{3}$ 	&time	&time&Err&time&Err&time&Err&time&Err
 		\\
 		\hline
 	 $1000 \times 1000 \times 100$  &11.66&9.60&0.0086&3.17&0.0092&45.25&0.0242&16.24&0.0441\\
 		\hline
 	 $1500 \times 1500 \times 100$   	&31.48&26.47&0.0108&7.45&0.0115&102.04&0.0265&41.08&0.0452       \\
 		\hline
 	 $2000 \times 2000 \times 100$  	&---	&67.67&0.0105&27.89&0.0111&168.73&0.0271&95.35&0.0456      \\		
 		\hline
 		 $1000 \times 1000 \times 300$  	&63.44&53.80&0.0125&23.98&0.0133&142.16&0.0337&75.33&0.0923\\
 		 \hline
 		  $1000 \times 1000 \times 400$  	&---	&---&---&30.33&0.0146&---&---&---&---\\
 		  \hline
 	\end{tabular}
 \end{table}

  \begin{figure}[htbp]
  	\centering
  	\subfigure[Time of Tensor 1]{
  		\includegraphics[width=7.5cm]{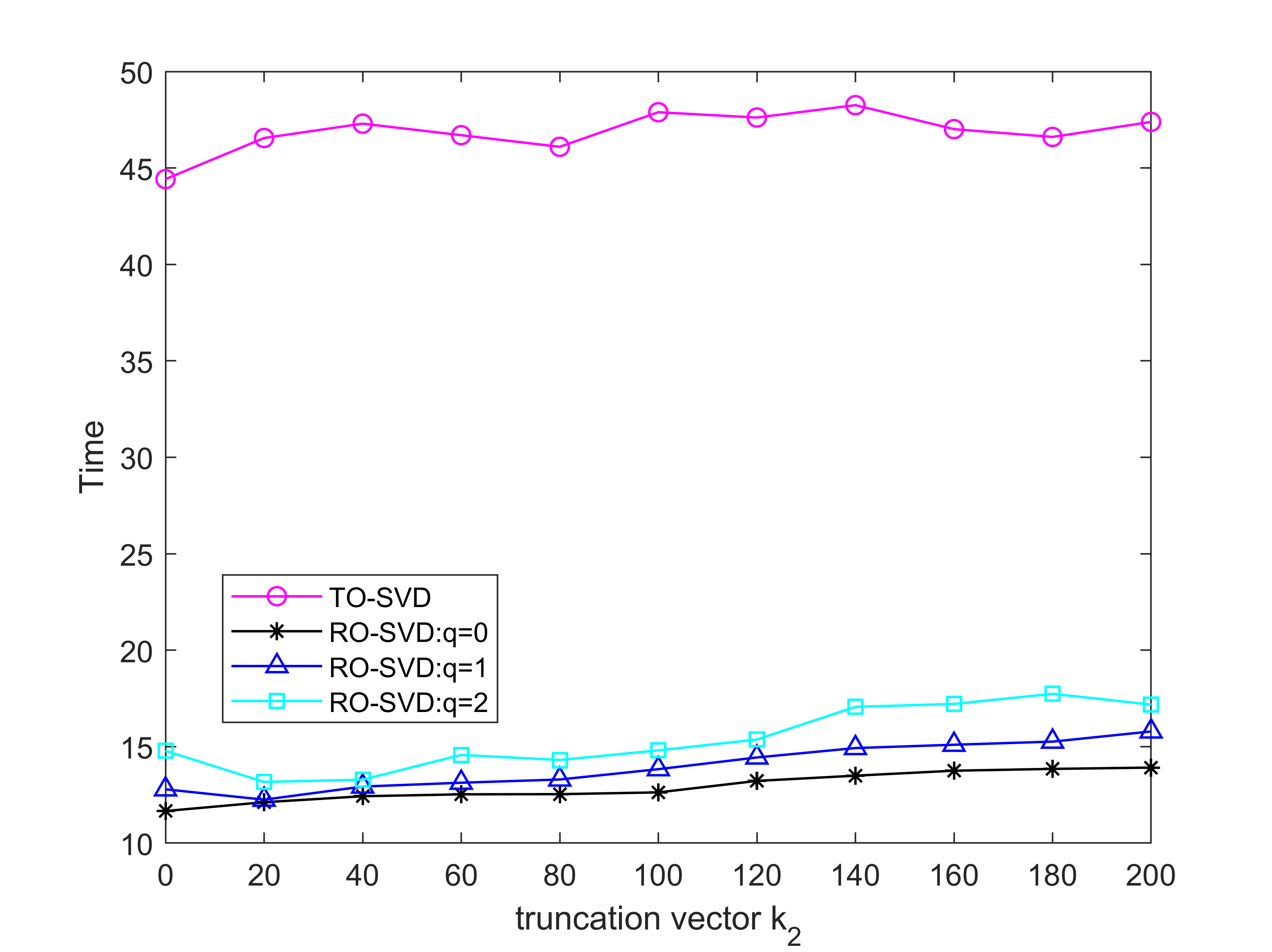}
  	}
  	\quad
  	\subfigure[Relative errors of Tensor 1 ]{
  		\includegraphics[width=7.5cm]{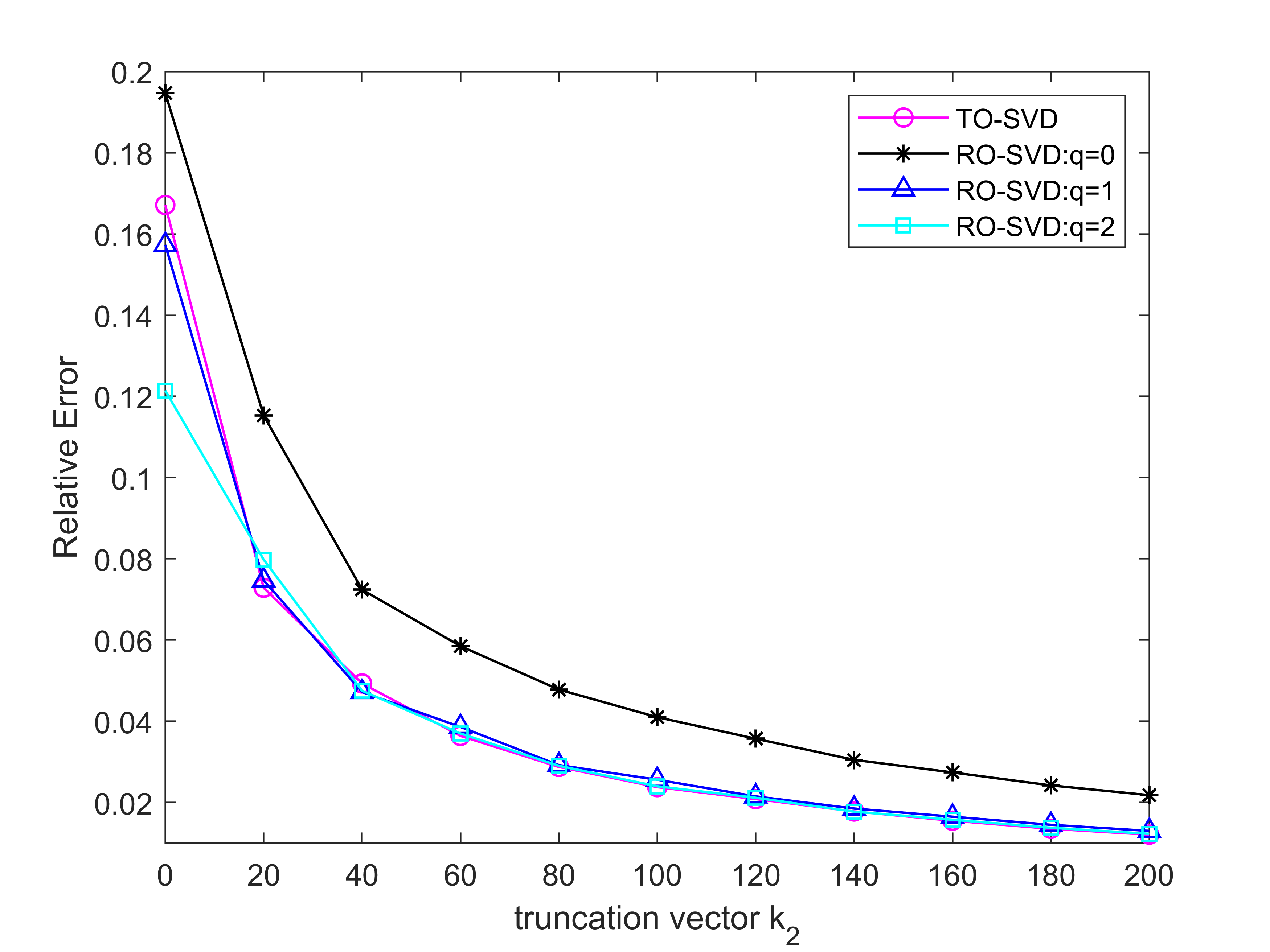}
  	}
  	\quad
  	\subfigure[Time of Tensor 2]{
  		\includegraphics[width=7.5cm]{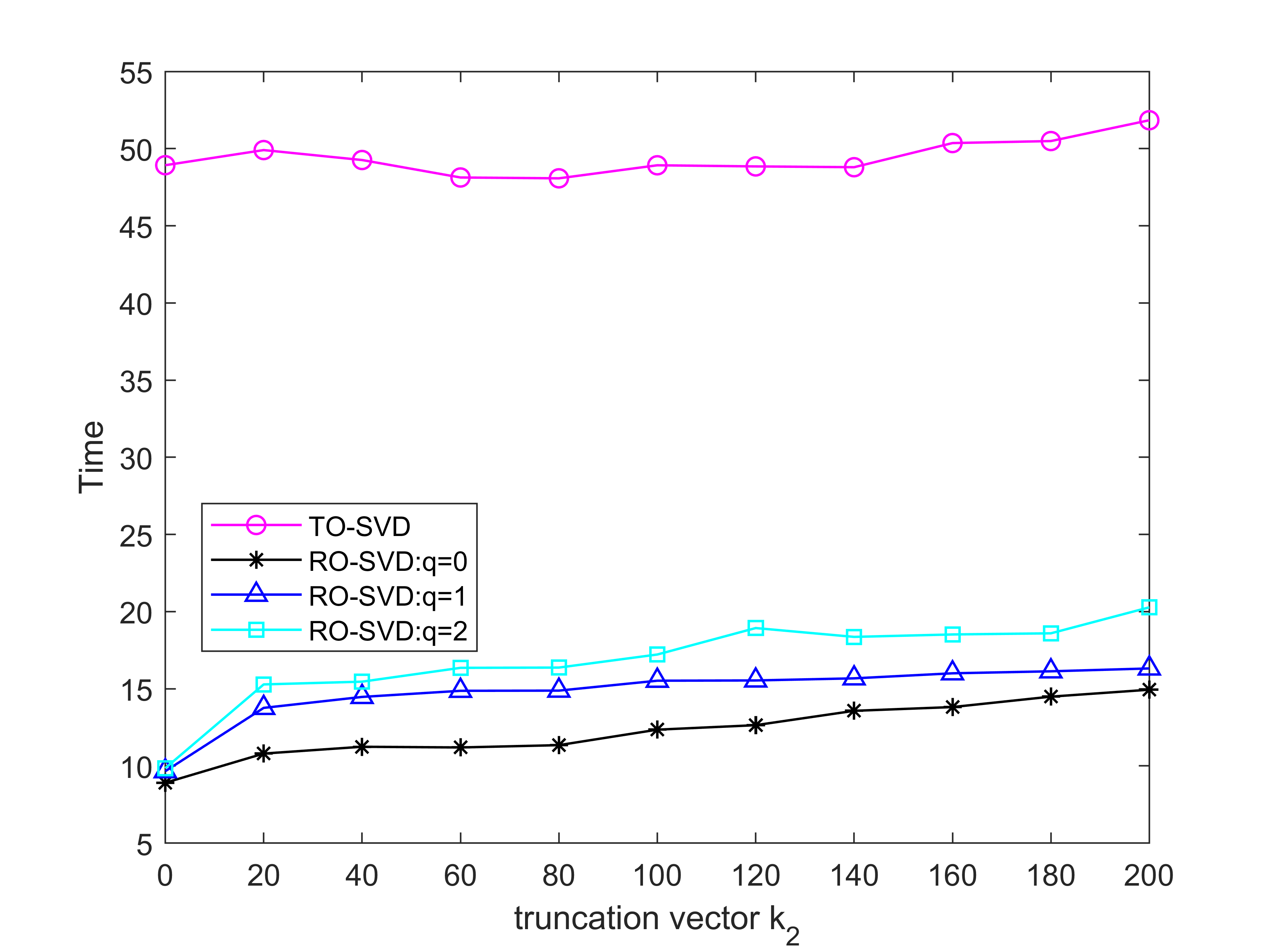}
  	}
  	\quad
  	\subfigure[Relative errors of Tensor 2]{
  		\includegraphics[width=7.5cm]{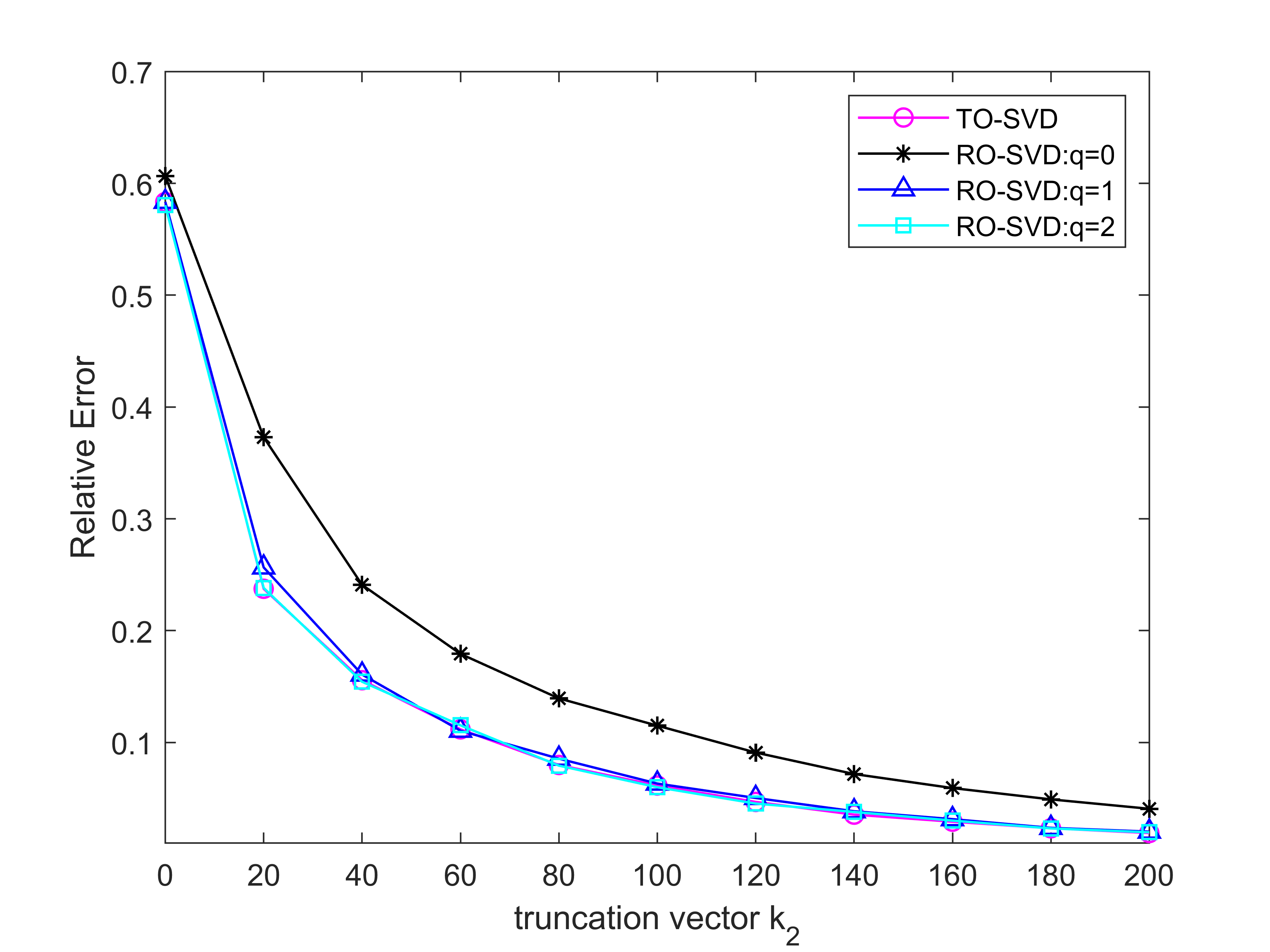}
  	}
  	\caption{Comparison results of the TO-SVD and RO-SVD on the test tensors ($k_{1} = 30$).}
  	\label{tensor12}
  \end{figure}

\begin{figure}[htbp]
	\centering
	\subfigure[Time of Tensor 1]{
		\includegraphics[width=7.5cm]{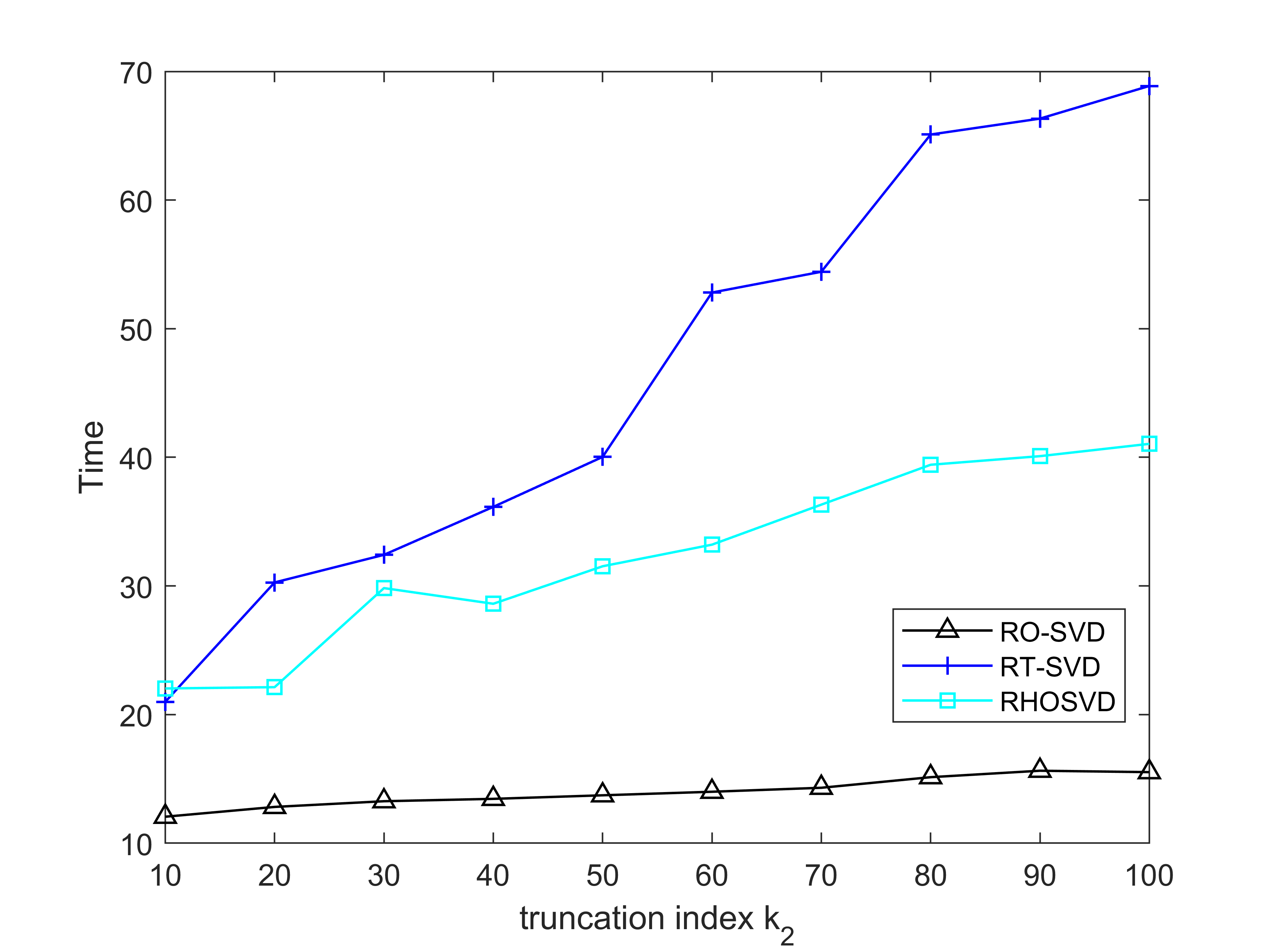}
	}
	\quad
	\subfigure[Relative errors of Tensor 1 ]{
		\includegraphics[width=7.5cm]{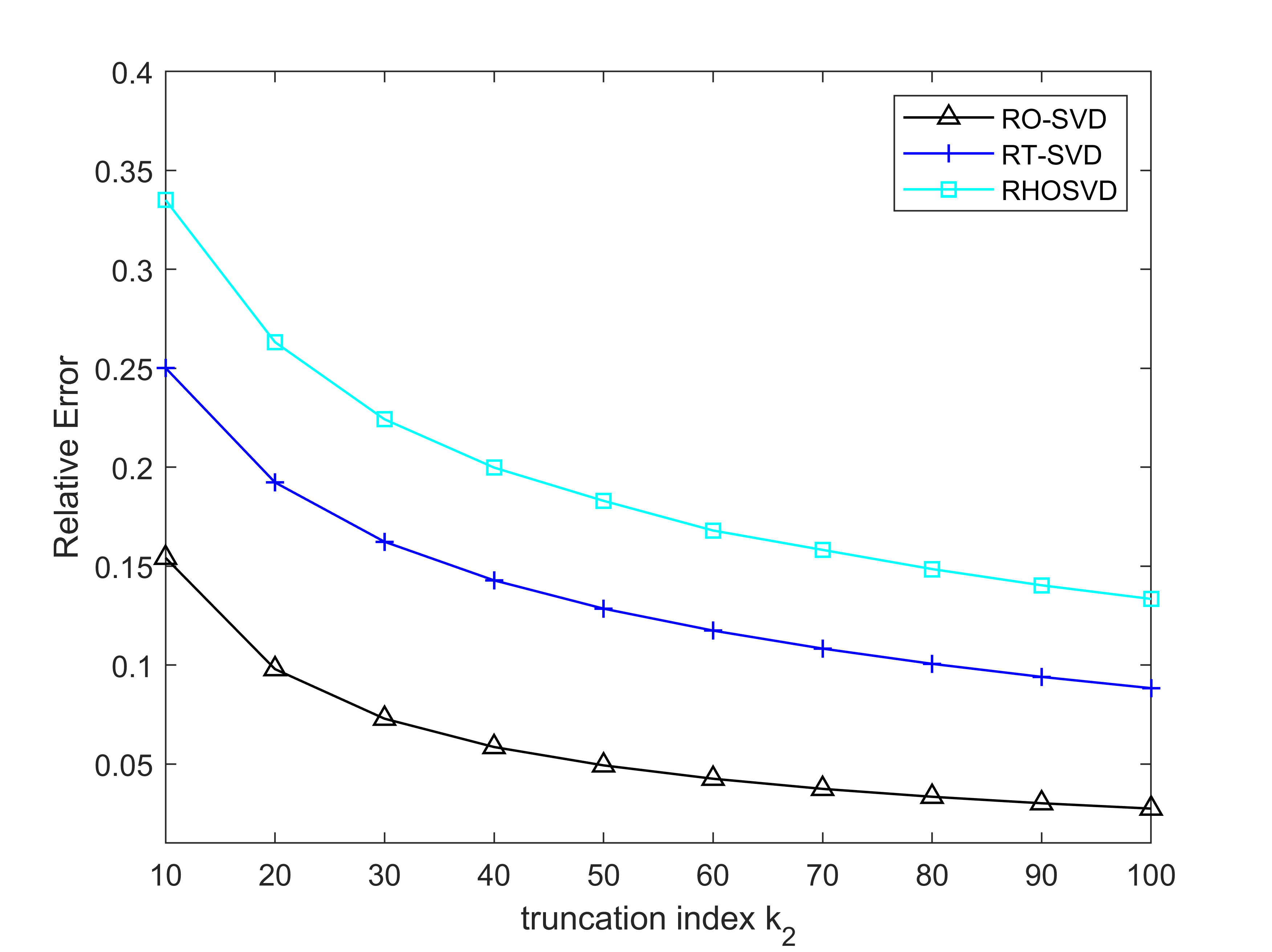}
	}
	\quad
	\subfigure[Time of Tensor 2]{
		\includegraphics[width=7.5cm]{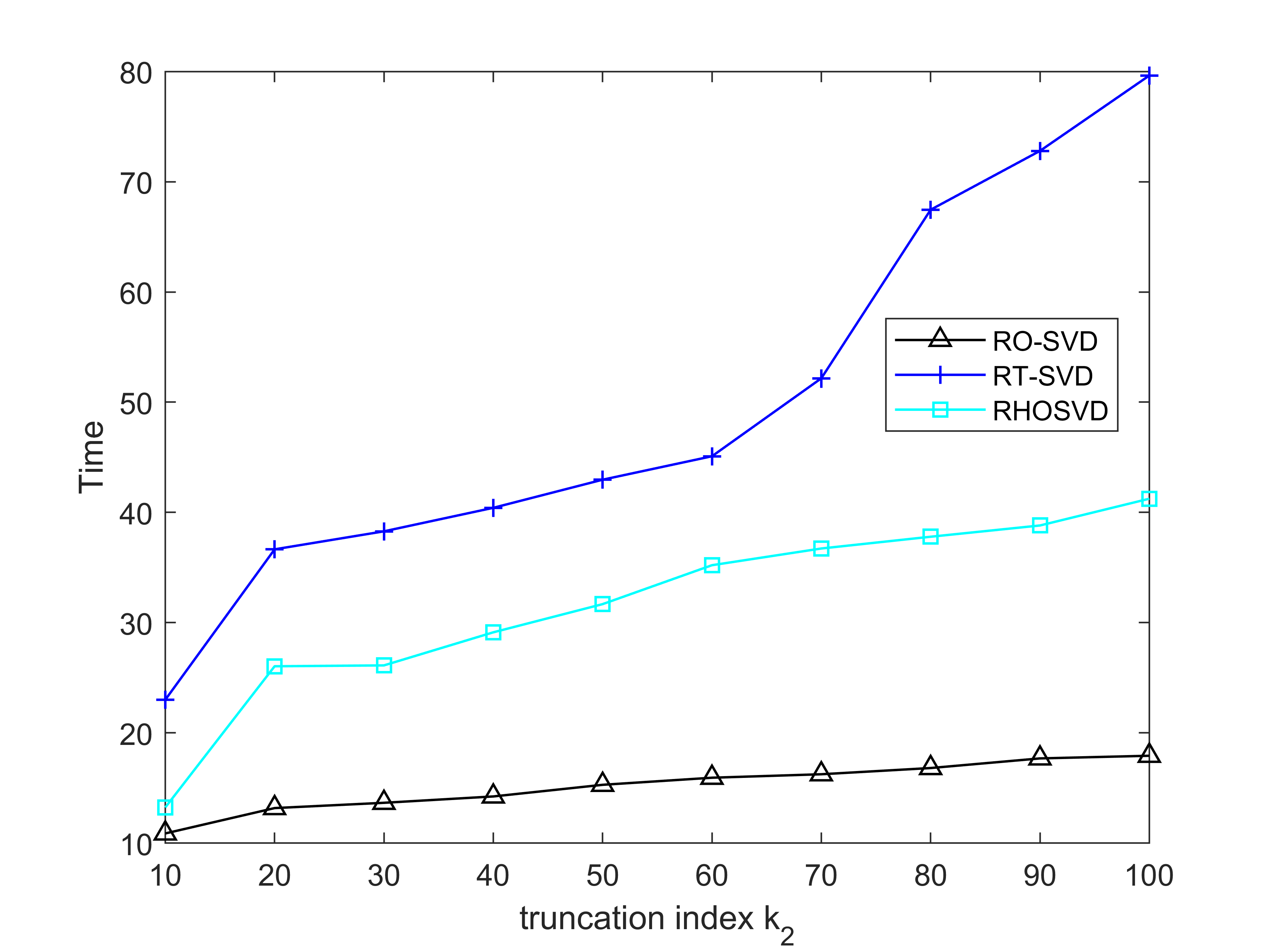}
	}
	\quad
	\subfigure[Relative errors of Tensor 2]{
		\includegraphics[width=7.5cm]{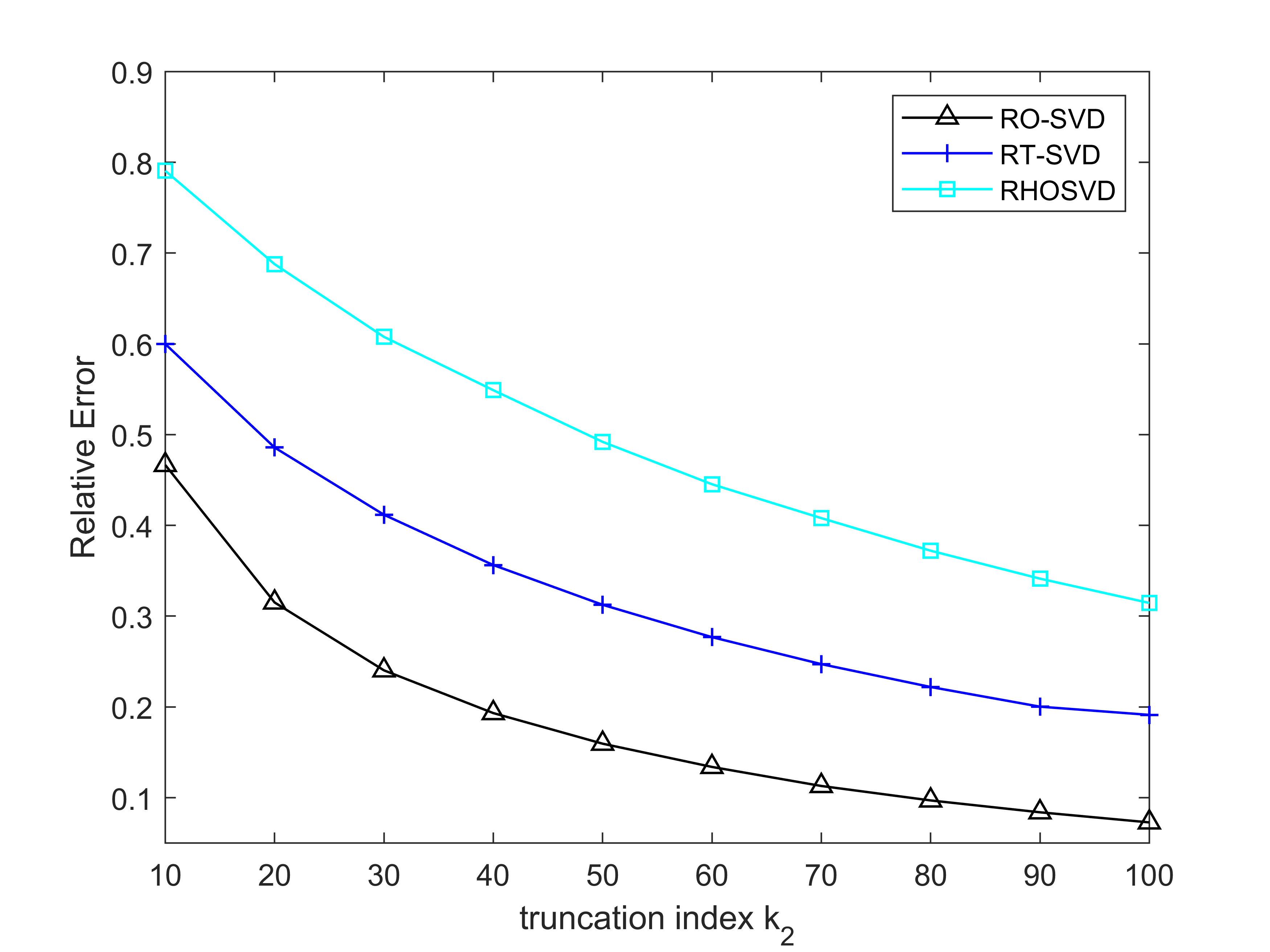}
	}
	\caption{Comparison results of three tensor randomized methods (RO-SVD, RT-SVD, and RHOSVD)
	on the test tensors ($k_{1} = 30$).}
	\label{tensor12OTH}
\end{figure}

\section{Conclusion} \label{Conclusion}
\hskip 2em The contributions of this paper are twofold: we revisit the O-SVD and refine its process from the viewpoint of the TTr1SVD and a truncated version for the O-SVD is given. Based on recent results on the randomized SVD template, we design a randomized algorithm for the O-SVD for third-order oriented tensors. In addition, we give the corresponding probabilistic error analysis of the RO-SVD. The performance of the RO-SVD is better than the RT-SVD for the same compression ratio and better than  the RHOSVD for the same size core tensor and shows great advantages in time cost if the tensor is well-oriented.

\hskip 2em  In the future, we will continue a further study of the adaptive randomized approach for the case where the target rank is unknown or cannot be estimated in advance. One
more potential research direction is the use of updated SVD algorithms \cite{fastbrand2006} to study dynamic video streaming.

	
  \section*{Acknowledgments}
  We would like to acknowledge the handling editor and three anonymous referees
for their useful comments and constructive suggestions which helped considerably to improve the quality of the paper.
  This work is supported by the National Natural Science Foundation of China (No. 12271108, 11801534),
  the Innovation Program of Shanghai Municipal Education Committee and  the Fundamental Research Funds for the Central Universities (No. 202264006).
%

\bibliographystyle{siam}
{\small
\bibliography{references}
}
\end{document}